\documentclass{article}%
\usepackage{amsmath}
\usepackage{amsfonts}
\usepackage{amssymb}
\usepackage{graphicx}%
\providecommand{\U}[1]{\protect\rule{.1in}{.1in}}
\newtheorem{theorem}{Theorem}
\newtheorem{acknowledgement}[theorem]{Acknowledgement}

\newtheorem{corollary}[theorem]{Corollary}

\newtheorem{definition}[theorem]{Definition}
\newtheorem{example}[theorem]{Example}

\newtheorem{lemma}[theorem]{Lemma}

\newtheorem{proposition}[theorem]{Proposition}

\newenvironment{proof}[1][Proof]{\noindent\textbf{#1.} }{\ \rule{0.5em}{0.5em}}

\numberwithin{equation}{section}

\begin{document}

\title{PBZ*-Lattices: \\Structure Theory and Subvarieties}
\author{Roberto Giuntini, Claudia Mure\c{s}an, Francesco Paoli
\and Dept. of Pedagogy, Psychology, Philosophy, University of Cagliari}
\maketitle

\begin{abstract}
We investigate the structure theory of the variety of \emph{PBZ*-lattices} and
some of its proper subvarieties. These lattices with additional structure
originate in the foundations of quantum mechanics and can be viewed as a
common generalisation of orthomodular lattices and Kleene algebras expanded by
an extra unary operation. We lay down the basics of the theories of ideals and
of central elements in PBZ*-lattices, we prove some structure theorems, and we
explore some connections with the theories of subtractive and binary
discriminator varieties.

\end{abstract}

\section{Introduction}

The papers \cite{GLP1+} and \cite{PBZ2} contain the beginnings of an algebraic
investigation of a variety of lattices with additional structure, the variety
{$\mathbb{PBZL}^{\ast}$ of }\emph{PBZ$^{\ast}$ --lattices}. The key motivation
for the introduction of this class of algebras comes from the foundations of
quantum mechanics. Consider the structure%
\[
\mathbf{E}\left(  \mathbf{H}\right)  =\left(  \mathcal{E}\left(
\mathbf{H}\right)  ,\wedge_{s},\vee_{s},^{\prime},^{\sim},\mathbb{O}%
,\mathbb{I}\right)  \text{,}%
\]
where:

\begin{itemize}
\item $\mathcal{E}\left(  \mathbf{H}\right)  $ is the set of all effects of a
given complex separable Hilbert space $\mathbf{H}$, i.e., positive linear
operators of $\mathbf{H}$ that are bounded by the identity operator
$\mathbb{I}$;

\item $\wedge_{s}$ and $\vee_{s}$ are the meet and the join, respectively, of
the \emph{spectral ordering} $\leq_{s}$ so defined for all $E,F\in
\mathcal{E}\left(  \mathbf{H}\right)  $:
\[
E\leq_{s}F\,\,\,\text{iff}\,\,\,\forall\lambda\in\mathbb{R}:\,\,M^{F}%
(\lambda)\leq M^{E}(\lambda),
\]
where for any effect $E$, $M^{E}$ is the unique spectral family \cite[Ch.
7]{Kr} such that $E=\int_{-\infty}^{\infty}\lambda\,dM^{E}(\lambda)$ (the
integral is here meant in the sense of norm-converging Riemann-Stieltjes sums
\cite[Ch. 1]{Strocco});

\item $\mathbb{O}$ and $\mathbb{I}$ are the null and identity operators, respectively;

\item $E^{\prime}=\mathbb{I}-E$ and $E^{\sim}=P_{\ker\left(  E\right)  }$ (the
projection onto the kernel of $E$).
\end{itemize}

The operations in $\mathbf{E}\left(  \mathbf{H}\right)  $ are well-defined.
The spectral ordering is indeed a lattice ordering \cite{Ols, deG} that
coincides with the usual ordering of effects induced via the trace functional
when both orderings are restricted to the set of projection operators of the
same Hilbert space.

A PBZ$^{\ast}$ --lattice{ can be viewed as an abstraction from this }concrete
physical model, much in the same way as an orthomodular lattice can be viewed
as an abstraction from a certain structure of projection operators in a
complex separable Hilbert space. The faithfulness of PBZ$^{\ast}$ --lattices
to the physical model whence they stem is further underscored by the fact that
they reproduce at an abstract level the \textquotedblleft
collapse\textquotedblright\ of several notions of \emph{sharp physical
property} that can be observed in $\mathbf{E}\left(  \mathbf{H}\right)  $.

Further motivation for the study of {$\mathbb{PBZL}^{\ast}$ comes from its
emerging relationships with many related algebraic structures (orthomodular
lattices, Kleene algebras, Stone algebras). In particular, PBZ$^{\ast}$
--lattices can be seen as a common generalisation of orthomodular lattices and
of Kleene algebras with an additional unary operation.}

This paper is devoted to laying down the basics of the structure theory of the
variety {$\mathbb{PBZL}^{\ast}$ and of some of its subvarieties; l}et us
briefly summarise its contents. In Section \ref{grosw} we dispatch a number of
preliminaries in order to keep the paper reasonably self-contained, including
a short r\'{e}sum\'{e} of the results in \cite{GLP1+} and \cite{PBZ2}. In
Section \ref{centros}, we study decompositions of PBZ$^{\ast}$ --lattices. As
it happens for orthomodular lattices, and more generally for members of all
Church varieties \cite{Sal}, direct decompositions in a PBZ$^{\ast}$ --lattice
$\mathbf{L}$ are induced by certain members of $L$ (the so-called
\emph{central elements}) that form a Boolean algebra and that can be
conveniently described. In particular, we show that the central elements in a
PBZ$^{\ast}$ --lattice $\mathbf{L}$ are those elements that \textquotedblleft
commute\textquotedblright\ with any $a\in L$, and that this \textquotedblleft
commuting\textquotedblright\ relation generalises the analogous relation of
decisive importance in the context of orthomodular lattices. In Section
\ref{Scipio}, w{e introduce the notion of a p-ideal (ideal closed under
perspectivity), mimicking the corresponding definition available for
orthomodular lattices. Although in the general case p-ideals lack many of the
strong properties one would expect from a reasonable notion of an ideal, as
soon as we zoom in on }the subvariety {$\mathbb{SDM}$ satisfying the strong De
Morgan law }$\left(  x\wedge y\right)  ^{\sim}\approx x^{\sim}\vee y^{\sim}${,
we can show that such ideals coincide with the $\mathbb{SDM}$-ideals in the
sense of Ursini (whence also with }$0$-classes of congruences, since
{$\mathbb{PBZL}^{\ast}$ and all its subvarieties are }$0$-subtractive). We
also prove that the $0$-assertional logic of {$\mathbb{SDM}$ is strongly
algebraisable and we characterise its equivalent variety semantics. Finally,
}we observe that the variety $V\left(  {\mathbb{AOL}}\right)  $ generated by
antiortholattices --- that is, {PBZ$^{\ast}$ --lattices with no nontrivial
sharp element} --- is a binary discriminator variety and we further simplify
the description of ideals in that case. In the concluding Section \ref{antis},
after streamlining the known equational basis for $V\left(  {\mathbb{AOL}%
}\right)  $, we axiomatise the varietal join of orthomodular lattices and the
variety generated by antiortholattices in the lattice of subvarieties of
{$\mathbb{PBZL}^{\ast}.$}

\section{Preliminaries\label{grosw}}

\subsection{Universal Algebra and Lattice Theory\label{ualglat}}

For basic information on universal algebra, the reader is referred to
\cite{bur, gralgu}.

Throughout this paper, all algebras will be nonempty; by a \emph{trivial
algebra} we will mean a one--element algebra, and a \emph{trivial variety}
will be a variety consisting solely of trivial algebras. If ${\mathbf{A}}$ is
an algebra, then $A$ will be the universe of ${\mathbf{A}}$; in some cases,
such as those of congruence lattices, lattices will be designated by their set
reducts. If ${\mathbb{V}}$ is a variety of algebras of similarity type $\nu$
and ${\mathbf{A}}$ or a reduct of ${\mathbf{A}}$ is a member of ${\mathbb{V}}%
$, then $(\mathrm{Con}_{{\mathbb{V}}}({\mathbf{A}}),\cap,\vee,\Delta
_{A},\nabla_{A})$ will be the bounded lattice of the congruences of
${\mathbf{A}}$ with respect to $\nu$; when ${\mathbb{V}}$ is the variety of
lattices, $\mathrm{Con}_{{\mathbb{V}}}({\mathbf{A}})$ will be denoted by
$\mathrm{Con}({\mathbf{A}})$. With ${\mathbb{V}}$ assumed implicit, the
congruence of $\mathbf{A}$ generated by an $S\subseteq A\times A$ will be
denoted by $Cg(S)$; for all $a,b\in A$, the principal congruence
$Cg(\{(a,b)\})$ will be denoted by $Cg(a,b)$.

For any lattice $\mathbf{L}$ and any $x,y\in L$, the principal filter (resp.
ideal) of $\mathbf{L}$ generated by $x$ will be denoted by $[x)$ (resp.
$(x]$), and, if $x\leq y$, then $[x,y]=[x)\cap(y]$ will be the interval of
$\mathbf{L}$ bounded by $x$ and $y$. The dual of any (bounded) lattice
$\mathbf{M}$ will be denoted by $\mathbf{M}^{d}$. If $\mathbf{A}$ is an
algebra with a bounded lattice reduct, then such a reduct will be indicated by
$\mathbf{A}_{l}$. In this case, a congruence $\theta$ of $\mathbf{A}$ (or any
of its reducts) is said to be \emph{pseudo-identical} iff $0^{\mathbf{A}%
}/\theta=\left\{  0^{\mathbf{A}}\right\}  $ and $1^{\mathbf{A}}/\theta
=\left\{  1^{\mathbf{A}}\right\}  $.

\subsection{PBZ$^{\ast}$ --lattices\label{pbz}}

We recap in this section some definitions and results on PBZ$^{\ast}$
--lattices (the latter mostly from \cite{GLP1+} and \cite{PBZ2}, except when
explicitly noted) that will be needed in the following.

\begin{definition}
\label{defbi}A \emph{bounded involution lattice} is an algebra $\mathbf{L}%
=(L,\wedge,\vee,^{\prime},0,1)$ of type $(2,2,1,0,0)$ such that $(L,\wedge
,\vee,0,1)$ is a bounded lattice with partial order $\leq$ and the following
conditions are satisfied for all $a,b\in L$:

\begin{itemize}
\item $a^{\prime\prime}=a$;

\item $a\leq b$ implies $b^{\prime}\leq a^{\prime}$.
\end{itemize}
\end{definition}

Note that, for any bounded involution lattice $\mathbf{L}$, the involution
$^{\prime}:L\rightarrow L$ is a dual lattice isomorphism of $\mathbf{L}_{l}$.

\begin{definition}
\label{de:kleene}A bounded involution lattice $\mathbf{L}=\left(
L,\wedge,\vee,^{\prime},0,1\right)  $ is a \emph{pseudo-Kleene algebra} in
case it satisfies any of the following two equivalent conditions:

\begin{enumerate}
\item for all $a,b\in L$, $\,$if$\,\,a\leq a^{\prime}\,\,$and$\,\,b\leq
b^{\prime},\,\,$then $\,a\leq b^{\prime}\,$;

\item for all $a,b\in L$, $a\wedge a^{\prime}\leq b\vee b^{\prime}$.
\end{enumerate}
\end{definition}

The class of bounded involution lattices is a variety, here denoted by
$\mathbb{BI}$. The involution of a pseudo-Kleene algebra is called
\emph{Kleene complement}. The variety of pseudo-Kleene algebras, for which see
e.g. \cite{PSK}, is denoted by $\mathbb{PKA}$. Distributive pseudo-Kleene
algebras are variously called \emph{Kleene lattices} or \emph{Kleene algebras}
in the literature. Observe that in \cite{GLP1+}, embracing the terminological
usage from \cite[p. 12]{RQT}, pseudo-Kleene algebras were referred to as
\textquotedblleft Kleene lattices\textquotedblright. In \cite{PBZ2}, however,
the authors switched to the less ambiguous \textquotedblleft pseudo-Kleene
algebras\textquotedblright.

In unsharp quantum logic, there are several competing purely algebraic
characterisations of sharp effects \cite[Ch. 7]{RQT}. A quantum effect or
property is usually called \emph{sharp} if it satisfies the noncontradiction principle:

\begin{definition}
\label{de:kleenesharp}Let $\mathbf{L}$ be a bounded involution lattice.

\begin{enumerate}
\item An element $a\in L$ is said to be \emph{Kleene-sharp} iff $a\wedge
a^{\prime}=0$. $S_{K}(\mathbf{L})\,\,$denotes the class of Kleene-sharp
elements of$\,\,\mathbf{L}$.

\item $\mathbf{L}$ is an \emph{ortholattice\/} iff $S_{K}(\mathbf{L})=L$.

\item $\mathbf{L}$ is an \emph{orthomodular lattice} iff $\mathbf{L}$ is an
ortholattice and, for all $a,b\in L$, $\,$if$\,\,\,a\leq b,\,\,\,$%
then$\,\,b=(b\wedge a^{\prime})\vee a$.
\end{enumerate}
\end{definition}

The variety of ortholattices is denoted by $\mathbb{OL}$. Among ortholattices,
orthomodular lattices\emph{ }play a crucial role in the standard (sharp)
approach to quantum logic. The class of orthomodular lattices is actually a
variety, hereafter denoted by $\mathbb{OML}$.

It is well-known that an ortholattice $\mathbf{L}$ is orthomodular if and only
if, for all $a,b\in L$, $\,$if$\,\,a\leq b\,\,\,$and$\,\,a^{\prime}\wedge
b=0,\,\,$then$\,\,a=b$. In the wider setting of bounded involution lattices,
the previous condition does not imply the stronger condition of
orthomodularity above. We will call this weaker condition
\emph{paraorthomodularity}.

\begin{definition}
\label{de:paraorthomodular}An algebra $\mathbf{L}$ with a bounded involution
lattice reduct is said to be \emph{paraorthomodular} iff, for all $a,b\in L$:
\[
\,\text{if}\,\,a\leq b\,\,\,\text{and}\,\,a^{\prime}\wedge b=0,\,\,\text{then}%
\,\,a=b.
\]

\end{definition}

It turns out that the class of paraorthomodular pseudo-Kleene algebras is a
proper quasivariety, whence we cannot help ourselves to the strong universal
algebraic properties that characterise varieties. It is then natural to wonder
whether there exists an \emph{expansion} of the language of bounded involution
lattices where the paraorthomodular condition can be equationally recovered.
The appropriate language expansion is provided by including an additional
unary operation and moving to the type $\left(  2,2,1,1,0,0\right)  $,
familiar in unsharp quantum logic from the investigation of
\emph{Brouwer-Zadeh lattices} (see \cite{CN} or \cite[Ch. 4.2]{RQT}).

\begin{definition}
\label{de:bazar}\noindent A \emph{Brouwer-Zadeh lattice\/} (or
\emph{BZ-lattice\/}) is an algebra%
\[
\mathbf{L}=\left(  {L,\,\wedge,\vee\,,\,}^{\prime}{,\,^{\sim},0\,,1}\right)
\]
of type $\left(  2,2,1,1,0,0\right)  $, such that:

\begin{enumerate}
\item $\left(  {L,\,\wedge,\vee\,,\,}^{\prime}{,0\,,1}\right)  $ is a
pseudo-Kleene algebra;

\item for all $a,b\in L$, the following conditions are satisfied:
\[%
\begin{array}
[c]{ll}%
\text{(1) }a\wedge a^{\sim}={0}\text{;} & \text{(2) }a\leq a^{\sim\sim
}\text{;}\\
\text{(3) }a\leq b\ \,\text{implies}\,\ b^{\sim}\leq a^{\sim}\text{;} &
\text{(4) }a{^{\sim\prime}=}a^{\sim\sim}\text{.}%
\end{array}
\]

\end{enumerate}
\end{definition}

The operation $^{\sim}$ is called the \emph{Brouwer complement} of the
BZ--lattice. The class of all BZ-lattices is a variety, denoted by
$\mathbb{BZL}$; $\mathbb{OL}$ can be identified with the subvariety of
$\mathbb{BZL}$ whose relative equational basis w.r.t. $\mathbb{BZL}$ is given
by the equation $x{\,^{\sim}=x}^{\prime}$. In any BZ-lattice, we set $\Diamond
x=x^{\sim\sim}$ and $\square x=x^{\prime\sim}$. The following arithmetical
lemma, the proof of which is variously scattered in the above-mentioned
literature and elsewhere \cite{Cat1, Cat2}, will be used without being
referenced in what follows.

\begin{lemma}
\label{basics}Let $\mathbf{L}$ be a BZ-lattice. For all $a,b\in L$, the
following conditions hold:%

\begin{tabular}
[c]{clcl}%
\textbf{(i)} & $a^{\sim\sim\sim}=a^{\sim}$; & \textbf{(vi)} & $\Box(a\wedge
b)=\Box a\wedge\Box b$;\\
\textbf{(ii)} & $a^{\sim}\leq a^{\prime}$; & \textbf{(vii)} & $\Diamond(a\vee
b)=\Diamond a\vee\Diamond b$;\\
\textbf{(iii)} & $(a\vee b)^{\sim}=a^{\sim}\wedge b^{\sim}$; & \textbf{(viii)}
& $\Diamond(a\wedge b)\leq\Diamond a\wedge\Diamond b$;\\
\textbf{(iv)} & $a^{\sim}\vee b^{\sim}\leq(a\wedge b)^{\sim}$; & \textbf{(ix)}
& if $a^{\prime}\leq a$, then $a^{\sim}=0$.\\
\textbf{(v)} & $(\Box(a^{\prime}))^{\prime}=\Diamond a$; &  &
\end{tabular}

\end{lemma}

We remarked above that Kleene-sharpness is not the unique purely algebraic
characterisation of a sharp quantum property. Two noteworthy alternatives now
become available in our expanded language of BZ-lattices.

\begin{definition}
\label{de:modalbrouwer} Let $\mathbf{L}$ be a BZ-lattice.

\begin{enumerate}
\item An element $a\in L$ is said to be \emph{$\Diamond$-sharp} iff
$a=\Diamond a$; the class of all $\Diamond$-sharp elements of $\mathbf{L}$
will be denoted by $S_{\Diamond}(\mathbf{L})$.

\item An element $a\in L$ is said to be \emph{Brouwer-sharp} iff $a\vee
a^{\sim}=1$; the class of all Brouwer-sharp elements of $\mathbf{L}$ will be
denoted by $S_{B}(\mathbf{L})$.
\end{enumerate}
\end{definition}

It is easy to derive from the previous lemma that, in any BZ-lattice
$\mathbf{L}$, $S_{\Diamond}(\mathbf{L})=\{a^{\sim}:a\in L\}=\{a\in
L:a^{\prime}=a^{\sim}\}$. For any BZ-lattice $\mathbf{L}$, we have that
$S_{\Diamond}(\mathbf{L})\subseteq S_{B}(\mathbf{L})\subseteq S_{K}%
(\mathbf{L})$. However, in any BZ-lattice of effects of a Hilbert space (under
the meet and join operation induced by the spectral ordering) these three
classes coincide. Consequently, it makes sense to investigate whether there is
a class of BZ-lattices for which this collapse result can be recovered at a
purely abstract level. The next definition and theorem answer this question in
the affirmative.

\begin{definition}
\label{de:bzlstar}A \emph{BZ$^{\ast}$-lattice} is a BZ-lattice $\mathbf{L}$
that satisfies, for all $a\in L$, the condition%
\[
(\ast)\noindent\qquad(a\wedge a^{\prime})^{\sim}\leq a^{\sim}\vee\square a.
\]

\end{definition}

\begin{theorem}
\label{co:collassone} Let $\mathbf{L}$ be a paraorthomodular BZ$^{\ast}%
$-lattice. Then,
\[
S_{\Diamond}(\mathbf{L})=S_{B}(\mathbf{L})=S_{K}(\mathbf{L}).
\]

\end{theorem}

As pleasing as this result may be, the class of paraorthomodular BZ$^{\ast}%
$-lattices still suffers from a major shortcoming: the paraorthomodularity
condition is quasiequational. However, the next result shows that it can be
replaced by an equation, so that paraorthomodular BZ$^{\ast}$-lattices form a
variety, which we will denote by $\mathbb{PBZL}^{\ast}$ and whose members will
be called, in brief, \emph{PBZ}$^{\ast}$\emph{-lattices}.

\begin{theorem}
\label{th:paradia} Let $\mathbf{L}$ be a BZ$^{\ast}$-lattice. The following
conditions are equivalent:

\begin{enumerate}
\item[(1)] $\mathbf{L}$ is paraorthomodular;

\item[(2)] $\mathbf{L}$ satisfies the following \emph{$\Diamond$%
-orthomodularity} condition for all $a,b\in L$:
\[
\,(a^{\sim}\vee(\Diamond a\wedge\Diamond b))\wedge\Diamond a\leq\Diamond b.
\]

\end{enumerate}
\end{theorem}

Every bounded lattice can be embedded as a sublattice into a PBZ$^{\ast}$
--lattice \cite[Lm. 5.3]{GLP1+}. Consequently, $\mathbb{PBZL}^{\ast}$
satisfies no nontrivial identity in the language of lattices.

The naturalness of the concept of a PBZ$^{\ast}$-lattice is further reinforced
by the circumstance that BZ-lattices of effects of a Hilbert space, under the
spectral ordering, qualify as instances of PBZ$^{\ast}$-lattices:

\begin{theorem}
{\ \label{concreto}Let $\mathbf{H}$ be a complex separable Hilbert space. The
algebra%
\[
\mathbf{E}\left(  \mathbf{H}\right)  =\left\langle \mathcal{E}\left(
\mathbf{H}\right)  ,\wedge_{s},\vee_{s},^{\prime},^{\sim},\mathbb{O}%
,\mathbb{I}\right\rangle ,
\]
(see the introduction for the notation) is a PBZ$^{\ast}$ --lattice.
Moreover,}%
\[
{S_{K}(\mathbf{E}\left(  \mathbf{H}\right)  )=}S_{\Diamond}(\mathbf{E}\left(
\mathbf{H}\right)  )=S_{B}(\mathbf{E}\left(  \mathbf{H}\right)  ){\ }%
\]
{is an orthomodular subuniverse of $\mathbf{E}\left(  \mathbf{H}\right)  $
consisting of all the projection operators of $\mathbf{H}$. }
\end{theorem}

All orthomodular lattices become, of course, {PBZ$^{\ast}$ --lattices when
endowed with a Brouwer complement that equals their Kleene complement. In
every PBZ$^{\ast}$ --lattice }${\mathbf{L}}$, {$S_{K}(\mathbf{L})$ is always
the universe of the largest orthomodular subalgebra $\mathbf{S}_{K}%
(\mathbf{L})$} of $\mathbf{L}$, so that $\mathbf{L}$ is orthomodular iff it
satisfies $x^{\sim}\approx x^{\prime}$. Further examples of {PBZ$^{\ast}$
--lattices are given by those algebras in this class that are
\textquotedblleft as far apart as possible\textquotedblright\ from
orthomodular lattices. In any orthomodular lattice }${\mathbf{L}}${,
$S_{K}({\mathbf{L}})=L$; on the other hand, by definition, a PBZ$^{\ast}$
--lattice $\mathbf{L}$ is an \emph{antiortholattice} iff $S_{K}({\mathbf{L}%
})=\left\{  0,1\right\}  $. We denote by $\mathbb{AOL}$ the class of
antiortholattices.}

\begin{lemma}
\label{barabba}{\ }

\begin{enumerate}
\item A PBZ$^{\ast}$ --lattice {$\mathbf{L}$ belongs to $\mathbb{AOL}$ iff
$0^{\sim}=1$ and, for all }$a\in L\setminus\{0\}$, $a^{\sim}=0$.

\item {Every $\mathbf{L}\in$ $\mathbb{AOL}$ is directly indecomposable.}

\item {$\mathbb{AOL}$ is a proper universal class. }
\end{enumerate}
\end{lemma}

The Brouwer complement of Lemma \ref{barabba}.(i) is called\emph{ trivial}.

For all $n\geq1$, the $n$-element Kleene chain with universe $D_{n}%
=\{0,d_{1},d_{2},\ldots,$\linebreak$d_{n-2},1\}$, with $0<d_{1}<d_{2}%
<\ldots<d_{n-2}<1${, is an antiortholattice $\mathbf{D}_{n}$ under the trivial
Brouwer complement.} To avoid notational overloading, the reduct
$(\mathbf{D}_{n})_{l}$ will simply be denoted by $\mathbf{D}_{n}$, as well.
Note that every finite chain is self--dual both as a bounded lattice and as a
Kleene algebra, so the notation $\mathbf{D}_{n}^{d}$ is superfluous in these
cases; the same can be stated about direct products of finite chains, in
particular about Boolean algebras.

The following easy results are observed (sometimes implicitly) in the
literature on BZ-lattices, in particular in \cite{GLP1+} and in \cite{PBZ2}:

\begin{lemma}
\label{klintoaol}

\begin{enumerate}
\item \label{l15(1)} Any pseudo-Kleene algebra, endowed with the trivial
Brouwer complement, becomes a BZ-lattice.

\item \label{l15(2)} Any paraorthomodular pseudo-Kleene algebra which, endowed
with the trivial Brouwer complement, satisfies condition $(\ast)$, becomes an antiortholattice.

\item \label{l15(3)} Any pseudo-Kleene algebra in which $0$ is
meet--irreducible is paraorthomodular and satisfies condition $(\ast)$ when
endowed with the trivial Brouwer complement, whence it becomes an antiortholattice.
\end{enumerate}

\label{l15}
\end{lemma}

We will repeatedly have the occasion to consider the following identities in
the language of BZ--lattices:

\begin{description}
\item[SDM] (the \emph{Strong de Morgan law}) $\left(  x\wedge y\right)
^{\sim}\approx x^{\sim}\vee y^{\sim}$;

\item[WSDM] (\emph{weak SDM}) $\left(  x\wedge y^{\sim}\right)  ^{\sim}\approx
x^{\sim}\vee\Diamond y$;

\item[DIST] $x\wedge\left(  y\vee z\right)  \approx\left(  x\wedge y\right)
\vee\left(  x\wedge z\right)  $;

\item[J2] $x\approx\left(  x\wedge\left(  y\wedge y^{\prime}\right)  ^{\sim
}\right)  \vee\left(  x\wedge\Diamond\left(  y\wedge y^{\prime}\right)
\right)  $;

\item[SK] $x\wedge\Diamond y\leq\square x\vee y$.
\end{description}

Clearly, SDM implies WSDM. Observe that $\mathbb{OML}$ satisfies SDM, J2 and
SK. Trivially, $\mathbb{AOL}$ satisfies WSDM and J2, whence $\mathbb{OML}\vee
V(\mathbb{AOL})$ satisfies these two identities.

We list some useful properties of the variety {$V\left(  \mathbb{AOL}\right)
$ generated by antiortholattices, including an axiomatisation relative to
$\mathbb{PBZL}^{\ast}$.}

\begin{lemma}
{\ \label{arismo}Let $\mathbf{A}\in$}$V\left(  {\mathbb{AOL}}\right)  $. Then
for all $a,b,c\in A$:

\hspace*{-6pt}%
\begin{tabular}
[c]{l}%
(i) $a\wedge\left(  b\vee c^{\sim}\right)  =\left(  a\wedge b\right)
\vee\left(  a\wedge c^{\sim}\right)  $;\\
(ii) $a\vee\left(  b\wedge c^{\sim}\right)  =\left(  a\vee b\right)
\wedge\left(  a\vee c^{\sim}\right)  $;\\
(iii) $a^{\sim}\wedge\left(  b\vee c\right)  =\left(  a^{\sim}\wedge b\right)
\vee\left(  a^{\sim}\wedge c\right)  $;\\
(iv) $a^{\sim}\vee\left(  b\wedge c\right)  =\left(  a^{\sim}\vee b\right)
\wedge\left(  a^{\sim}\vee c\right)  \text{.}$%
\end{tabular}

\end{lemma}

\begin{theorem}
\label{caniggia}

\begin{enumerate}
\item An equational basis for $V\left(  \mathbb{AOL}\right)  $ relative to
$\mathbb{PBZL}^{\ast}$ is given by the identities%
\begin{align*}
\text{(AOL1) }  &  \left(  x^{\sim}\vee y^{\sim}\right)  \wedge\left(
\Diamond x\vee z^{\sim}\right)  \approx\left(  \left(  x^{\sim}\vee y\right)
\wedge\left(  \Diamond x\vee z\right)  \right)  ^{\sim}\text{;}\\
\text{(AOL2) }  &  x\approx\left(  x\wedge y^{\sim}\right)  \vee\left(
x\wedge\Diamond y\right)  \text{;}\\
\text{(AOL3) }  &  x\approx\left(  x\vee y^{\sim}\right)  \wedge\left(
x\vee\Diamond y\right)  \text{.}%
\end{align*}

\item Every subdirectly irreducible member of $V(\mathbb{AOL})$ is an antiortholattice.
\end{enumerate}
\end{theorem}

Clearly $V\left(  \mathbb{AOL}\right)  \cap\mathbb{OML}$ is the variety
{$\mathbb{BA}$ of Boolean algebras}.

The lattice {$\mathbf{L}_{\mathbb{PBZL}^{\ast}}$ }of{ subvarieties of
$\mathbb{PBZL}^{\ast}$ has $\mathbb{BA}$ as a unique atom. It is well-known
that $\mathbb{BA}$ has a single orthomodular cover \cite[Cor. 3.6]{BH}: the
variety $V\left(  \mathbf{MO}_{2}\right)  $, generated by the simple modular
ortholattice with $4$ atoms. Moreover:}

\begin{theorem}
\label{mobutu}There is a single non-orthomodular cover of {$\mathbb{BA}$ in
$\mathbf{L}_{\mathbb{PBZL}^{\ast}}$, the variety }$V\left(  {\mathbf{D}_{3}%
}\right)  $ generated by the $3$-element antiortholattice chain, whose
equational basis relative to $V\left(  \mathbb{AOL}\right)  $ is given by the
identity SK.
\end{theorem}

Two other notable subvarieties of $V\left(  \mathbb{AOL}\right)  $ are the
variety {$\mathbb{DIST}$}, whose equational basis relative to $V\left(
\mathbb{AOL}\right)  $ (or, equivalently, relative to $\mathbb{PBZL}%
^{\mathbb{\ast}}$) is given by the distribution identity DIST, {and the
variety $\mathbb{SAOL}$, whose }equational basis relative to $V\left(
\mathbb{AOL}\right)  $ is given by the Strong De Morgan identity SDM. We have that:

\begin{theorem}
\label{antonello}$V\left(  \mathbf{D}_{5}\right)  =\mathbb{DIST}%
\cap\;\mathbb{SAOL}$.
\end{theorem}

A more circumscribed study of $\mathbb{DIST}$ and its subvarieties, also
focussing on the relationship with known classes of algebras (including
Kleene-Stone algebras \L ukasiewicz algebras, Heyting-Wajsberg algebras) is
currently in preparation \cite{GLPf}

\subsection{Subtractive Varieties}

Subtractive varieties were introduced by Ursini \cite{OSV1} to enucleate the
common features of pointed varieties with a good ideal theory, like groups,
rings or Boolean algebras. They were further investigated in \cite{OSV2, OSV3,
OSV4, OSV5}.

\begin{definition}
\label{subtra}Let $\mathbb{V}$ a variety of type $\nu$, and let $0$ be a
nullary term (or equationally definable constant) of type $\nu$. $\mathbb{V}$
is called $0$\emph{-subtractive} if there exists a binary term $s$, also of
type $\nu$, s.t. $\mathbb{V}$ satisfies the identities $s\left(  x,x\right)
\approx0$ and $s\left(  x,0\right)  \approx x.$ A variety of type $\nu$ which
is $0$-subtractive w.r.t. at least one constant $0$ of type $\nu$ is called
\emph{subtractive} tout court.
\end{definition}

It is not hard to see that subtractivity is a congruence property: namely, a
variety $\mathbb{V}$ is $0$-subtractive exactly when in each $\mathbf{A}%
\in\mathbb{V}$ congruences permute at $0$ (meaning that for all $\theta
,\varphi$ in $\mathrm{Con}_{{\mathbb{V}}}\left(  \mathbf{A}\right)  $,
$0^{\mathbf{A}}/(\theta\circ\varphi)=0^{\mathbf{A}}/(\varphi\circ\theta)$).

To investigate ideals in this context, first and foremost, we need a workable
general notion of ideal encompassing all the intended examples mentioned above
(normal subgroups of groups, two-sided ideals of rings, ideals or filters of
Boolean algebras). Ursini's candidate for playing this role is defined below.

\begin{definition}
\label{ideal term}

\begin{enumerate}
\item If $\mathbb{K}$ is a class of similar algebras whose type $\nu$ is as in
Definition \ref{subtra}, then a term $p\left(  \overrightarrow{x}%
,\overrightarrow{y}\right)  $ of type $\nu$ is a $\mathbb{K}$\emph{-ideal
term} in $\overrightarrow{x}$ iff $\mathbb{K\vDash}p\left(
0,...,0,\overrightarrow{y}\right)  \approx0$.

\item A nonempty subset $J$ of the universe of an $\mathbf{A}\in\mathbb{K}$ is
a $\mathbb{K}$\emph{-ideal} of $\mathbf{A}$ (w.r.t. $0$) iff for any
$\mathbb{K}$-ideal term $p\left(  \overrightarrow{x},\overrightarrow{y}%
\right)  $ in $\overrightarrow{x}$ we have that $p^{\mathbf{A}}\left(
\overrightarrow{a},\overrightarrow{b}\right)  \in J$ whenever
$\overrightarrow{a}\in J$ and $\overrightarrow{b}\in A$.
\end{enumerate}
\end{definition}

We will denote by $\mathcal{I}_{\mathbb{K}}\left(  \mathbf{A}\right)  $ the
set (or the lattice) of all $\mathbb{K}$-ideals of $\mathbf{A}$, dropping the
subscript whenever $\mathbb{K}$ can be contextually identified; observe that
$\left\{  0\right\}  ,A\in\mathcal{I}_{\mathbb{K}}\left(  \mathbf{A}\right)
$. The main reason that backs our previous claim to the effect that
subtractive varieties have a good ideal theory is given by the following
result. Let $\mathbb{V}$ be a variety of type $\nu$. Recall that an algebra
$\mathbf{A}$ from ${\mathbb{V}}$ is said to be $0$\emph{-regular} iff the map
sending a congruence $\theta\in\mathrm{Con}\left(  \mathbf{A}\right)  $ to its
$0$-class $0^{\mathbf{A}}/\theta$ is injective. The variety ${\mathbb{V}}$ is
said to be $0$\emph{-regular} iff every $\mathbf{A}\in\mathbb{V}$ is
$0$-regular; this happens exactly when there exists a finite family of binary
$\nu$-terms (called \emph{Fichtner terms}) $\left\{  d_{i}\left(  x,y\right)
\right\}  _{i\leq n}$ such that $\mathbb{V\vDash}d_{1}\left(  x,y\right)
\approx0\&...\&d_{n}\left(  x,y\right)  \approx0\Leftrightarrow x\approx y$.

\begin{theorem}
\label{funziona}

\begin{enumerate}
\item Subtractive varieties have normal ideals. That is, if $\mathbb{V}$ is a
$0$-subtractive variety and $\mathbf{A}\in\mathbb{V}$, then
\[
\mathcal{I}_{\mathbb{V}}\left(  \mathbf{A}\right)  =\left\{  I\subseteq
A:I=0^{\mathbf{A}}/\theta\text{ for some }\theta\in\mathrm{Con}_{{\mathbb{V}}%
}\left(  \mathbf{A}\right)  \right\}  .
\]

\item If $\mathbb{V}$ is a $0$-subtractive and $0$-regular variety, then, for
every $\mathbf{A}\in\mathbb{V}$, $\mathrm{Con}_{{\mathbb{V}}}\left(
\mathbf{A}\right)  $ is isomorphic to $\mathcal{I}_{\mathbb{V}}\left(
\mathbf{A}\right)  $.
\end{enumerate}
\end{theorem}

Actually, the situation described by the previous theorem can be made more
precise as follows. Let $\mathbf{A}$ be an algebra in a $0$-subtractive
variety $\mathbb{V}$. Then the following maps are well-defined: $\cdot
^{\delta},\cdot^{\varepsilon}:\mathcal{I}_{{\mathbb{V}}}(\mathbf{A}%
)\rightarrow\mathrm{Con}_{{\mathbb{V}}}(\mathbf{A})$, for all $I\in
\mathcal{I}_{{\mathbb{V}}}(\mathbf{A})$,
\[%
\begin{array}
[c]{ll}%
I^{\delta} & =\bigwedge\{\theta\in\mathrm{Con}_{{\mathbb{V}}}\left(
\mathbf{A}\right)  :0^{\mathbf{A}}/\theta=I\},\\
I^{\varepsilon} & =\bigvee\{\theta\in\mathrm{Con}_{{\mathbb{V}}}\left(
\mathbf{A}\right)  :0^{\mathbf{A}}/\theta=I\}.
\end{array}
\]

Henceforth, all unnecessary superscripts will be dropped for the sake of
conciseness. Note that, for all $I\in\mathcal{I}_{{\mathbb{V}}}(\mathbf{A})$,
we have $0/I^{\delta}=0/I^{\varepsilon}=I$, so that $I^{\delta}=\min
\{\theta\in\mathrm{Con}_{{\mathbb{V}}}\left(  \mathbf{A}\right)
:0/\theta=I\}$ and $I^{\varepsilon}=\max\{\theta\in\mathrm{Con}_{{\mathbb{V}}%
}\left(  \mathbf{A}\right)  :0/\theta=I\}$. Moreover, the map $I\mapsto\lbrack
I^{\delta},I^{\varepsilon}]$ is a lattice isomorphism from $\mathcal{I}%
_{{\mathbb{V}}}(\mathbf{A})$ to the following lattice of intervals of
$\mathrm{Con}_{{\mathbb{V}}}(\mathbf{A})$: $\{[\min(C_{\theta}),\max
(C_{\theta})]:\theta\in\mathrm{Con}_{{\mathbb{V}}}(\mathbf{A})\}$, where, for
all $\theta\in\mathrm{Con}_{{\mathbb{V}}}(\mathbf{A})$, $C_{\theta}%
=\{\alpha\in\mathrm{Con}_{{\mathbb{V}}}(\mathbf{A}):0/\alpha=0/\theta
\}=\{\alpha\in\mathrm{Con}_{{\mathbb{V}}}(\mathbf{A}):0/\theta\in A/\alpha\}$,
which is a complete sublattice of $\mathrm{Con}_{{\mathbb{V}}}(\mathbf{A})$.
Clearly, if $\mathbb{V}$ is in addition $0$-regular, then $C_{\theta}%
=\{\theta\}$ for all $\theta\in\mathrm{Con}_{{\mathbb{V}}}(\mathbf{A})$, hence
all intervals of the form $\left[  I^{\delta},I^{\varepsilon}\right]  $ for
some $I\in\mathcal{I}_{{\mathbb{V}}}(\mathbf{A})$ are trivial, and Theorem
\ref{funziona}.(ii) follows as a special case.

\begin{definition}
Let $\mathbf{A}$ be an algebra in a $0$-subtractive variety $\mathbb{V}$.
$\mathbf{A}$ is said to be \emph{reduced} iff $\left\{  0\right\}
^{\varepsilon}=\Delta_{A}$.
\end{definition}

The class of all reduced algebras in $\mathbb{V}$ will be denoted by
$\mathbb{V}_{\varepsilon}$. Clearly, for all $\theta\in\mathrm{Con}%
_{\mathbb{V}}\left(  \mathbf{A}\right)  $, we have $\mathbf{A}/\theta
\in\mathbb{V}_{\varepsilon}$ iff $\theta=I^{\varepsilon}$ for some
$I\in\mathcal{I}_{\mathbb{V}}\left(  \mathbf{A}\right)  $.

For a $0$-subtractive variety, a property that generalises $0$-regularity is
finite congruentiality. Roughly put, a $0$-subtractive variety is finitely
congruential if it has a family of terms that do \textquotedblleft part of the
job\textquotedblright\ usually dispatched by the Fichtner terms for $0$-regularity.

\begin{definition}
\label{granata} A variety $\mathbb{V}$, whose type $\mathcal{\nu}$ is as in
Definition \ref{subtra}, is \emph{finitely congruential} iff there exists a
finite set $\left\{  d_{i}\left(  x,y\right)  \right\}  _{i\leq n}$ of binary
$\mathcal{\nu}$-terms s.t., whenever $\mathbf{A}\in\mathbb{V}$ and
$I\in\mathcal{I}_{\mathbb{V}}(\mathbf{A})$, we have:
\[
I^{\varepsilon}=\{\left(  a,b\right)  :d_{i}^{\mathbf{A}}\left(  a,b\right)
\in I\text{ for all }i\leq n\}.
\]

\end{definition}

The next results establish a significant connection between the theory of
subtractive varieties and \emph{abstract algebraic logic} (for information on
this research area, the reader is referred to \cite{Font}). It turns out that,
for a $0$-subtractive variety $\mathbb{V}$, the properties of the
$0$-assertional logic of $\mathbb{V}$ yield relevant information on the
properties of the class of reduced algebras in $\mathbb{V}$, and conversely.

\begin{theorem}
\label{crnjar} If $\mathbf{A}$ is a member of a $0$-subtractive variety
$\mathbb{V}$, then $\mathcal{I}_{{\mathbb{V}} }\left(  \mathbf{A}\right)  $ is
the class of all deductive filters on $\mathbf{A}$ of the $0$-assertional
logic of $\mathbb{V}$, and for all $I\in\mathcal{I}_{\mathbb{V}}\left(
\mathbf{A}\right)  $, $I^{\varepsilon}=\Omega^{\mathbf{A}}\left(  I\right)  $.
\end{theorem}

\begin{theorem}
\cite[Thm. 3.12]{OSV3} \label{mistretta}For a $0$-subtractive variety
$\mathbb{V}$ the following are equivalent:

\begin{enumerate}
\item The $0$-assertional logic of $\mathbb{V}$ is equivalential.

\item $\mathbb{V}_{\varepsilon}$ is closed under subalgebras and direct products.
\end{enumerate}
\end{theorem}

\begin{theorem}
\cite[Thm. 3.16]{OSV3} \label{mistrettone}For a $0$-subtractive variety
$\mathbb{V}$ the following are equivalent:

\begin{enumerate}
\item The $0$-assertional logic of $\mathbb{V}$ is strongly algebraisable with
$\mathbb{V}_{\varepsilon}$ as an equivalent algebraic semantics.

\item $\mathbb{V}_{\varepsilon}$ is a variety.
\end{enumerate}
\end{theorem}

Important examples of subtractive varieties are \emph{binary discriminator
varieties}. Recall that a \emph{discriminator variety} \cite{Werner} is a
variety $\mathbb{V}$ of given type $\nu$ for which there exists a ternary
$\nu$-term $t(x,y,z)$ that realises the ternary discriminator function%
\[
t\left(  a,b,c\right)  =\left\{
\begin{array}
[c]{l}%
c\text{ if }a=b\text{,}\\
a\text{, otherwise}%
\end{array}
\right.
\]
on any subdirectly irreducible member of $\mathbb{V}$ (equivalently, on any
member of some class $\mathbb{K}$ such that $\mathbb{V}=V\left(
\mathbb{K}\right)  $). The introduction of \emph{binary discriminator
varieties} by Chajda, Hala\v{s}, and Rosenberg \cite{CHR} was aimed at
singling out an appropriate weakening of the ternary discriminator that can
vouchsafe some of the strong properties of discriminator varieties (like
congruence distributivity or congruence permutability) only \textquotedblleft
locally\textquotedblright, i.e. at $0$. Binary discriminator varieties are
paramount among subtractive varieties with \emph{equationally definable
principal ideals} \cite{OSV2}; they were thoroughly studied in the unpublished
\cite{BS}.

\begin{definition}
\cite{CHR}\label{zanzibar} Let $A$ be a nonempty set and fix $0\in A$. The
$0$-\emph{binary discriminator} on $A$ is the binary function $b_{0}^{A}$ on
$A$ defined by:%
\[
b_{0}^{A}\left(  a,c\right)  =\left\{
\begin{array}
[c]{l}%
a\text{ if }c=0\text{,}\\
0\text{ otherwise.}%
\end{array}
\right.
\]

An algebra $\mathbf{A}$ with a term definable element $0$ is said to be a
$0$-\emph{binary discriminator algebra} in case the $0$-binary discriminator
$b_{0}^{A}$ on $A$ is a term operation on $\mathbf{A}$. A variety $V\left(
\mathbb{K}\right)  $ is a $0$-\emph{binary discriminator variety} if it is
generated by a class $\mathbb{K}$ of $0$-binary discriminator algebras such
that the property is witnessed by the same terms for all members of
$\mathbb{K}$.
\end{definition}

\section{Central Elements\label{centros}}

One of the most distinctive and far-reaching chapters in the theory of
orthomodular lattices is the study of the commuting relation and of central
elements (see e.g. {\cite[\S \ 2]{BH}). Given an orthomodular lattice
}$\mathbf{L}$ and $a,b\in L$, $a$ is said to \emph{commute} with $b$ in case
$(a\wedge b)\vee(a^{\prime}\wedge b)=b$. Such a relation is reflexive and
symmetric. An element $a\in L$ is said to be \emph{central} in $\mathbf{L}$ in
case it commutes with all elements of $L$. The next celebrated result is one
of the most useful tools for practicioners of the field:

\begin{theorem}
[Foulis-Holland]\label{fuli}\cite[Prop. 2.8]{BH} If $\mathbf{L}$ is an
orthomodular lattice and $a,b,c\in L$ are such that $a$ commutes both with $b$
and with $c$, then the set $\left\{  a,b,c\right\}  $ generates a distributive
sublattice of $\mathbf{L}_{l}$.
\end{theorem}

Although these investigations were carried out in the special context of
orthomodular lattices, the notion of a central element is deeply rooted in
universal algebra. Recall that, if $\mathbf{A}$ is an algebra in a
double-pointed variety $\mathbb{V}$ with constants $0,1$, an element $e\in A$
\ is \emph{central} in $\mathbf{A}$ in case the congruences $Cg\left(
e,0\right)  $ and $Cg\left(  e,1\right)  $ are complementary factor
congruences{ of }$\mathbf{A}$ {\cite{Sal}}. {By }$C(\mathbf{A})${\ we denote
the \emph{centre}\ of $A$, i.e. the set of central elements of the algebra
$\mathbf{A}$.} In particular, if {$\mathbf{A}$ is a }\emph{Church algebra
}{\cite{Sal}, namely, if }there is an \textquotedblleft
if-then-else\textquotedblright\ term operation $q^{\mathbf{A}}$\ on
$\mathbf{A}$ s.t., for all $a,b\in A$, $q^{\mathbf{A}}\left(  1^{\mathbf{A}%
},a,b\right)  =a\emph{\ }$and$\emph{\ }q^{\mathbf{A}}\left(  0^{\mathbf{A}%
},a,b\right)  =b$, then, b{y defining
\[
x\wedge y=q(x,y,0)\text{, }x\vee y=q(x,1,y)\text{ and }x^{\prime}=q(x,0,1),
\]
we get: }

\begin{theorem}
{\ \label{centrobolla}\cite[Thm. 3.7]{Sal} The algebra $c\left[
\mathbf{A}\right]  =$}$\left(  {C(\mathbf{A}),\wedge,\vee,^{\prime}%
,0,1}\right)  ${ is a Boolean algebra which is isomorphic to the Boolean
algebra of factor congruences of $\mathbf{A}$. }
\end{theorem}

In Church algebras{, central elements can be equationally characterised
\cite[Prop. 3.6]{Sal}. When studying a specific variety $\mathbb{V}$, however,
more informative descriptions of central elements in members of $\mathbb{V}$
can sometimes be found either in terms of certain properties of intrinsic
interest, or by appropriately streamlining the above-mentioned equational
characterisation. As regards members of }$\mathbb{OML}$, for example, it can
be shown that the two definitions of central element we have given above are
equivalent. For $\mathbb{PBZL}^{\ast}$, on the other hand, a more economical
equational description of central elements\textup{ was provided }in
\textup{\cite[Lm. 5.9]{GLP1+}. }We reproduce this lemma for the reader`s convenience:

\begin{lemma}
\label{freccetta} Let $\mathbf{L}\in\mathbb{PBZL}^{\ast}$. Then $e\in L$ is
central in $\mathbf{L}$ iff it satisfies the following conditions for all
$a,b\in L$:

\begin{description}
\item[C1] $a=(e\wedge a)\vee(e^{\prime}\wedge a)$;

\item[C2] $(a\vee b)\wedge e=(a\wedge e)\vee(b\wedge e)$;

\item[C3] $(a\vee b)\wedge e^{\prime}=(a\wedge e^{\prime})\vee(b\wedge
e^{\prime})$;

\item[C4] $((e\vee b)\wedge(e^{\prime}\vee a))^{\sim}=(e\vee b^{\sim}%
)\wedge(e^{\prime}\vee a^{\sim})$.
\end{description}
\end{lemma}

The aim of the next subsection is to improve on this result, giving a
description of the centre in a generic PBZ$^{\ast}$ --lattice that resembles
as closely as possible the one we have in orthomodular lattices.

\subsection{Central Elements in PBZ$^{\ast}$ --lattices}

Throughout the rest of this subsection, unless mentioned otherwise,
$\mathbf{L}$ will be an arbitrary PBZ$^{\ast}$ --lattice. Observe that
$C(\mathbf{L})\subseteq S_{K}(\mathbf{L})$, by C1 for $a=1$. Since $e^{\prime
}=e^{\sim}$ for all $e\in S_{K}(\mathbf{L})$, it follows that, for all $e\in
L$: $e\in C(\mathbf{L})$ iff $e^{\sim}\in C(\mathbf{L})$. In {\cite[Thm.
5.4]{GLP1+} i}t is also proved{ that central elements in any member
}$\mathbf{L}$ {of }$V\left(  \mathbb{AOL}\right)  $ are exactly the members of
$S_{K}\left(  \mathbf{L}\right)  $. The problem of finding a manageable
description of central elements in PBZ$^{\ast}$ --lattices, that relates in a
perspicuous way to the notions of a centre and of a commutator in orthomodular
lattice, was however left open. We now set about filling this gap. We show
that the property of being central, in a generic PBZ$^{\ast}$ --lattice, is
actually two-sided. In order to belong to $C(\mathbf{L})$, an $e\in L$ should
not only \textquotedblleft commute\textquotedblright\ with any element of $L$,
in a sense of commuting that appropriately generalises the corresponding
(reflexive and symmetric) relation on orthomodular lattices; but it should
also be such that, for any $a\in L$, $(e\wedge a)^{\sim}=e^{\sim}\vee a^{\sim
}$ and $(e^{\prime}\wedge a)^{\sim}=\square e\vee a^{\sim}$. This latter
component is sort of \textquotedblleft hidden\textquotedblright\ in the case
of $\mathbb{OML}$, where the Strong De Morgan identity is satisfied across the board.

We define the following binary relations on $L$:%
\[%
\begin{tabular}
[c]{l}%
$C_{\mathbf{L}}=\{\left(  a,b\right)  \in L^{2}:(a\wedge b)\vee(a^{\prime
}\wedge b)=b\}$;\\
$C_{\text{SDM},\mathbf{L}}=C_{\mathbf{L}}\cap\{\left(  a,b\right)  \in
L^{2}:(a\wedge b)^{\sim}=a^{\sim}\vee b^{\sim}\text{ and }(a^{\prime}\wedge
b)^{\sim}=\square a\vee b^{\sim}\}$.
\end{tabular}
\ \
\]

Both relations are clearly reflexive. Also, for all $a\in L$, $(1,a)\in
C_{\text{SDM},\mathbf{L}}\subseteq C_{\mathbf{L}}$, while the following four
conditions are mutually equivalent: i) $(a,1)\in C_{\text{SDM},\mathbf{L}}$;
ii) $(a,1)\in C_{\mathbf{L}}$; iii) $a\vee a^{\prime}=1$; iv) $a\in
S_{K}({\mathbf{L}})$. Hence, $C_{\text{SDM},\mathbf{L}}$ is symmetric exactly
when $C_{\mathbf{L}}$ is symmetric, which in turn obtains if and only if
${\mathbf{L}}$ is orthomodular --- in which case, apparently, the two
relations coincide.

\begin{definition}
Let {$\mathbf{L}\in\mathbb{PBZL}^{\ast}$}. An element $a\in L$ is said to be
\emph{PBZ$^{\ast}$ --central} iff $\left(  a,b\right)  \in C_{\text{SDM}%
,\mathbf{L}}$ for all $b\in L$.
\end{definition}

We will denote by $C_{pbz}(\mathbf{L})$ the set of all PBZ$^{\ast}$ --central
elements of $\mathbf{L}$:%
\[
C_{pbz}(\mathbf{L})=\{a\in L:(\forall\,b\in L)\,((a,b)\in C_{\text{SDM}%
,\mathbf{L}})\}.
\]
We also consider the following subset of $L$:%
\[
C_{p}(\mathbf{L})=\{a\in L:(\forall\,b\in L)\,((a,b)\in C_{\mathbf{L}})\}.
\]
In virtue of the above,
\[%
\begin{array}
[c]{l}%
\hspace*{-5pt}C_{pbz}(\mathbf{L})=C_{p}(\mathbf{L})\cap\{a\in L:\!(\forall
\,b\in L)\left(  \,(a\wedge b)^{\sim}\!=a^{\sim}\vee b^{\sim},(a^{\prime
}\wedge b)^{\sim}\!=\square a\vee b^{\sim}\right)  \}\\
\hspace*{-5pt}=C_{p}(\mathbf{L})\cap\{a\in S_{K}({\mathbf{L}}):(\forall\,b\in
L)\,\left(  (a\wedge b)^{\sim}=a^{\sim}\vee b^{\sim},(a^{\prime}\wedge
b)^{\sim}=\square a\vee b^{\sim}\right)  \}\\
\hspace*{-5pt}\subseteq C_{p}(\mathbf{L})\subseteq S_{K}({\mathbf{L}})\text{.}%
\end{array}
\]

\begin{lemma}
\label{goldrush}Let ${\mathbf{L}}\in\mathbb{PBZL}^{\ast}$. Then:

\begin{enumerate}
\item \label{goldrush1} $C_{p}(\mathbf{L})=\{a\in S_{K}({\mathbf{L}}%
):(\forall\,b\in L)\,((a\wedge b)\vee(a^{\sim}\wedge b)=b)\}=\{a\in
S_{K}({\mathbf{L}}):(\forall\,b\in L)\,((a\vee b)\wedge(a^{\sim}\vee b)=b)\}$;

\item \label{goldrush2} for all $a\in S_{K}({\mathbf{L}})$, $a\in
C_{p}(\mathbf{L})$ iff $a^{\sim}\in C_{p}(\mathbf{L})$; furthermore, $a\in
C_{pbz}(\mathbf{L})$ iff $a^{\sim}\in C_{pbz}(\mathbf{L})$;

\item \label{goldrush3} if ${\mathbf{L}}$ satisfies WSDM, then $C_{pbz}%
(\mathbf{L})=C_{p}(\mathbf{L})$.
\end{enumerate}
\end{lemma}

\begin{proof}
Let $a\in S_{K}({\mathbf{L}})$, arbitrary.

(\ref{goldrush1}) The fact that $a^{\prime}=a^{\sim}$ gives us the first
equality. To obtain the second equality, notice that, for all $b\in L$, we
have: $(a,b)\in C_{\mathbf{L}}$ iff $(a\wedge b)\vee(a^{\prime}\wedge b)=b$
iff $(a^{\prime}\vee b^{\prime})\wedge(a\vee b^{\prime})=b^{\prime}$ iff
$(a\vee b^{\prime})\wedge(a^{\sim}\vee b^{\prime})=b^{\prime}$, hence:
$(a,b)\in C_{\mathbf{L}}$ for all $b\in L$ iff $(a\vee b)\wedge(a^{\sim}\vee
b)=b$ for all $b\in L$.

(\ref{goldrush2}) From (\ref{goldrush1}), along with the fact that $a=\Diamond
a$.

(\ref{goldrush3}) From the fact that $C_{p}(\mathbf{L})\subseteq
S_{K}({\mathbf{L}})$.
\end{proof}

Note from this proof that $C_{\mathbf{L}}$ and $C_{\text{SDM},\mathbf{L}}$
preserve the Kleene complement and that, under WSDM, the relation
$C_{\text{SDM},\mathbf{L}}\cap(S_{K}({\mathbf{L}})\times L)$ $=$
$C_{\mathbf{L}}\cap(S_{K}({\mathbf{L}})\times L)$ preserves the Brouwer
complement. Some useful properties of members of $C_{p}(\mathbf{L})$ follow in
the next two lemmas.

\begin{lemma}
\label{andlatd}Let {$\mathbf{L}\in\mathbb{PBZL}^{\ast}$} and let $a\in L$ be
such that $a{^{\sim}}\in C_{p}(\mathbf{L})$. Then, for all $b,c\in L$:

\begin{enumerate}
\item \label{andlatd1} $a^{\sim}\vee b=a^{\sim}\vee(\Diamond a\wedge b)$ and
$a^{\sim}\wedge b=a^{\sim}\wedge(\Diamond a\vee b)$;

\item \label{andlatd2} $b\vee(c\wedge\Diamond a)=(b\vee c\vee a^{\sim}%
)\wedge(b\vee\Diamond a)$ and $b\wedge(c\vee\Diamond a)=(b\wedge c\wedge
a^{\sim})\vee(b\wedge\Diamond a)$;

\item \label{andlatd3} $a^{\sim}\wedge(b\vee c)=a^{\sim}\wedge(b\vee(a^{\sim
}\wedge c))$ and $a^{\sim}\vee(b\wedge c)=a^{\sim}\vee(b\wedge(a^{\sim}\vee
c))$.
\end{enumerate}
\end{lemma}

\begin{proof}
By adapting the corresponding proofs in \cite[Lm. 5.10]{GLP1+}.
\end{proof}

\begin{lemma}
\label{ottolini}Let $\mathbf{L}\in\mathbb{PBZL}^{\ast}$, and let $a\in L$ be
such that $a^{\sim}\in C_{p}(\mathbf{L})$. Then the following hold for all
$b,c\in L$:

\begin{enumerate}
\item \label{ottolini1} $b\vee(c\wedge a^{\sim})=(b\vee c)\wedge(b\vee
a^{\sim})$ and $b\wedge(c\vee a^{\sim})=(b\wedge c)\vee(b\wedge a^{\sim})$;

\item \label{ottolini2} $a^{\sim}\wedge(b\vee c)=(a^{\sim}\wedge
b)\vee(a^{\sim}\wedge c)$ and $a^{\sim}\vee(b\wedge c)=(a^{\sim}\vee
b)\wedge(a^{\sim}\vee c)$.
\end{enumerate}
\end{lemma}

\begin{proof}
We prove the first equalities in each pair. The lattice duals are derived similarly.

(\ref{ottolini1}) Let $x=b\vee\left(  c\wedge a^{\sim}\right)  $. Then:%
\[%
\begin{array}
[c]{lll}%
x & =x\wedge\left(  x\vee\left(  c\wedge\left(  x\vee\Diamond a\right)
\right)  \right)  & \text{Absorption}\\
& =\left(  x\vee a^{\sim}\right)  \wedge\left(  x\vee\Diamond a\right)
\wedge\left(  x\vee\left(  c\wedge\left(  x\vee\Diamond a\right)  \right)
\right)  & \text{Lemma \ref{goldrush}}\\
& =\left(  x\vee a^{\sim}\right)  \wedge\left(  x\vee\left(  c\wedge\left(
x\vee\Diamond a\right)  \right)  \right)  & \text{Lattice prop.}\\
& =\left(  b\vee a^{\sim}\right)  \wedge\left(  x\vee\left(  c\wedge\left(
x\vee\Diamond a\right)  \right)  \right)  & \text{Absorption}\\
& =\left(  b\vee a^{\sim}\right)  \wedge\left(  x\vee c\right)  & \\
& =\left(  b\vee c\right)  \wedge\left(  b\vee a^{\sim}\right)  &
\text{Absorption}%
\end{array}
\]

The penultimate equality is obtained by observing that, according to Lemma
\ref{andlatd}.(\ref{andlatd1}), since $\Diamond(a^{\sim})=a^{\sim}$, we have:
$c\leq b\vee c\vee\Diamond a=b\vee\Diamond a\vee(a^{\sim}\wedge c)=x\vee
\Diamond a$.

(\ref{ottolini2})
\[%
\begin{array}
[c]{lll}%
{a^{\sim}\wedge\left(  b\vee c\right)  } & =a^{\sim}\wedge\left(  b\vee\left(
a^{\sim}\wedge c\right)  \right)  & \text{Lemma \ref{andlatd}.(\ref{andlatd3}%
)}\\
& =\left(  \left(  a^{\sim}\wedge c\right)  \vee a^{\sim}\right)
\wedge\left(  b\vee\left(  a^{\sim}\wedge c\right)  \right)  &
\text{Absorption}\\
& =\left(  a^{\sim}\wedge b\right)  \vee\left(  a^{\sim}\wedge c\right)  &
\text{(\ref{ottolini1})}%
\end{array}
\]

\end{proof}

\begin{theorem}
\label{centravanti} Let $\mathbf{L}\in\mathbb{PBZL}^{\ast}$. Then
$C_{pbz}(\mathbf{L})=C(\mathbf{L})$.
\end{theorem}

\begin{proof}
Let $e\in L$.

Assume that $e^{\sim}\in C_{pbz}(\mathbf{L})\subseteq C_{p}(\mathbf{L})$. We
verify all the conditions C1 through C4 in Lemma \ref{freccetta}. C1 holds
because $(e^{\sim},a)\in C_{\text{SDM},\mathbf{L}}\subseteq C_{\mathbf{L}}$
for all $a\in L$. C2 holds by Lemma \ref{ottolini}.(\ref{ottolini2}). C3 holds
by Lemma \ref{ottolini}.(\ref{ottolini2}) and Lemma \ref{goldrush}, which
ensures us that $e^{\sim\prime}=\Diamond e\in C_{p}(\mathbf{L})$. From the
latter fact, the equality $e^{\sim}\vee\Diamond e=1$ and Lemma \ref{ottolini}%
.(i), we obtain:%
\begin{align*}
a\wedge b  &  \leq\!(a\vee b)\wedge(a\vee\Diamond e)\wedge(b\vee e^{\sim})\\
&  =\!(a\vee b)\wedge(e^{\sim}\vee b)\wedge(\Diamond e\vee a)\wedge(e^{\sim
}\vee\Diamond e)\\
&  =(b\vee(a\wedge e^{\sim}))\wedge(\Diamond e\vee(a\wedge e^{\sim}))\\
&  =(a\wedge e^{\sim})\vee(b\wedge\Diamond e),
\end{align*}
whence, also using the equality $e^{\sim}\wedge\Diamond e=0$ and Lemma
\ref{goldrush}.(ii),
\[%
\begin{array}
[c]{lll}%
\left(  \left(  e^{\sim}\vee b\right)  \wedge\left(  \Diamond e\vee a\right)
\right)  ^{\sim} & =\left(  \left(  e^{\sim}\wedge a\right)  \vee\left(
\Diamond e\wedge b\right)  \vee\left(  b\wedge a\right)  \right)  ^{\sim} &
\text{Lemma \ref{ottolini}}\\
& =\left(  \left(  e^{\sim}\wedge a\right)  \vee\left(  \Diamond e\wedge
b\right)  \right)  ^{\sim} & \\
& =\left(  e^{\sim}\wedge a\right)  ^{\sim}\wedge\left(  \Diamond e\wedge
b\right)  ^{\sim} & \\
& =\left(  e^{\sim}\vee b^{\sim}\right)  \wedge\left(  \Diamond e\vee a^{\sim
}\right)  & e^{\sim},\Diamond e\in C_{pbz}(\mathbf{L})
\end{array}
\]

Conversely, if {$e^{\sim}$ is central in $\mathbf{L}$, then by C1 we have that
for all }$a\in L$, $a=\left(  a\wedge e{^{\sim}}\right)  \vee\left(
a\wedge\Diamond e\right)  $, whereby {$e^{\sim}$}$\in C_{p}(\mathbf{L})$. Now,
recall that if $e^{\sim}$ is central, then so is $\Diamond e$, whereby C4
holds for both elements --- i.e., for all $a,b\in L$:%
\begin{align*}
\text{(i) }((e^{\sim}\vee b)\wedge(\Diamond e\vee a))^{\sim}  &  =\left(
e^{\sim}\vee b_{{}}^{\sim}\right)  \wedge(\Diamond e\vee a^{\sim})\text{;}\\
\text{(ii) }\left(  (\Diamond e\vee b)\wedge(e^{\sim}\vee a)\right)  ^{\sim}
&  =\left(  \Diamond e\vee b_{{}}^{\sim}\right)  \wedge\left(  e^{\sim}\vee
a^{\sim}\right)  \text{.}%
\end{align*}
Letting $b=e^{\sim}$ in (i) and using Lemma \ref{ottolini}.(\ref{ottolini2}),
we have that $\left(  e{^{\sim}}\wedge a\right)  ^{\sim}=\left(  \left(
e{^{\sim}}\vee e^{\sim}\right)  \wedge\left(  \Diamond e\vee a\right)
\right)  ^{\sim}=\Diamond e\vee a^{\sim}$. Similarly, letting $b=\Diamond e$
in (ii), we obtain $\left(  \Diamond e\wedge a\right)  ^{\sim}=e{^{\sim}}\vee
a^{\sim}$.
\end{proof}

\subsection{Central Elements in the Variety Generated by Antiortholattices}

I{n any member }$\mathbf{L}$ {of }$V\left(  \mathbb{AOL}\right)  ${ central
elements }are exactly the sharp elements: $C\left(  \mathbf{L}\right)
=S_{K}\left(  \mathbf{L}\right)  $. Therefore, for any $a\in\mathbf{L}$,
$\mathbf{L}$ is decomposable as $\mathbf{L}/Cg\left(  a^{\sim},0\right)
\times\mathbf{L}/Cg\left(  a^{\sim},1\right)  $, or equivalently as
$\mathbf{L}/Cg\left(  a^{\sim},0\right)  \times\mathbf{L}/Cg\left(  \Diamond
a,0\right)  $. However, the mentioned results in {\cite{Sal} do not provide us
with a uniform recipe to obtain an explicit description of the factors in this
decomposition. This situation marks a sharp contrast with the case of
orthomodular lattices, where the following result is available:}

\begin{theorem}
{\cite[Lm. 2.7]{BH} }Let $\mathbf{L}\in\mathbb{OML}$ and let $e\in C\left(
\mathbf{L}\right)  $. Then $\mathbf{L}\simeq\mathbf{L}_{1}\times\mathbf{L}%
_{2}$, where $\mathbf{L}_{1},\mathbf{L}_{2}$ are algebras whose universes are
the intervals $\left[  0,e\right]  $ and $\left[  0,e^{\prime}\right]  $, respectively.
\end{theorem}

{In this subsection, we similarly characterise the factors in these
decompositions in terms of algebras on intervals in }$\mathbf{L}$.

\begin{lemma}
\label{nilla}Let $\mathbf{L}\in V\left(  \mathbb{AOL}\right)  $ and let $a\in
L$. Define:%
\begin{align*}
\mathbf{L}_{1}  &  =\left(  \left[  0,a^{\sim}\right]  ,\wedge,\vee,^{\prime
1},^{\sim1},0,a^{\sim}\right)  \text{;}\\
\mathbf{L}_{2}  &  =\left(  \left[  0,\Diamond a\right]  ,\wedge,\vee
,^{\prime2},^{\sim2},0,\Diamond a\right)  \text{,}%
\end{align*}
where for all $b$ in $\left[  0,a^{\sim}\right]  $, $b^{\prime1}=b^{\prime
}\wedge a^{\sim}$, $b^{\sim1}=b^{\sim}\wedge a^{\sim}$, while for all $c$ in
$\left[  0,\Diamond a\right]  $, $c^{\prime2}=c^{\prime}\wedge\Diamond a$,
$c^{\sim2}=c^{\sim}\wedge\Diamond a$. Then the algebras $\mathbf{L}_{1}$ and
$\mathbf{L}_{2}$ are in $V\left(  \mathbb{AOL}\right)  $.
\end{lemma}

\begin{proof}
Clearly, $\mathbf{L}_{1}$ and $\mathbf{L}_{2}$ are bounded lattices. We now
verify the remaining properties for $\mathbf{L}_{1}$; by replacing $a$ by
$a^{\sim}$, we obtain our claim for $\mathbf{L}_{2}$.

($\mathbf{L}_{1}$ is a pseudo-Kleene algebra). Let $b,c\in\left[  0,a^{\sim
}\right]  $. Then, using Lemma \ref{arismo}.(iii), $b^{\prime1\prime1}=\left(
b^{\prime}\wedge a^{\sim}\right)  ^{\prime}\wedge a^{\sim}=\left(
b\vee\Diamond a\right)  \wedge a^{\sim}=b\wedge a^{\sim}=b$. Moreover,
$\left(  b\wedge c\right)  ^{\prime1}=\left(  b\wedge c\right)  ^{\prime
}\wedge a^{\sim}=\left(  b^{\prime}\vee c^{\prime}\right)  \wedge a^{\sim
}=\left(  b^{\prime}\wedge a^{\sim}\right)  \vee\left(  c^{\prime}\wedge
a^{\sim}\right)  =b^{\prime1}\vee c^{\prime1}$. Finally, since $b\wedge
b^{\prime}\leq$ $c\vee c^{\prime}$ in $\mathbf{L}$, we use Lemma
\ref{arismo}.(ii) and obtain%
\[
b\wedge b^{\prime1}=b\wedge b^{\prime}\wedge a^{\sim}\leq\left(  c\vee
c^{\prime}\right)  \wedge\left(  c\vee a^{\sim}\right)  =c\vee\left(
c^{\prime}\wedge a^{\sim}\right)  =c\vee c^{\prime1}\text{.}%
\]

($\mathbf{L}_{1}$ is paraorthomodular). Let $b,c\in\left[  0,a^{\sim}\right]
$. Suppose that $b\leq c$ and that $b^{\prime1}\wedge c=b^{\prime}\wedge
a^{\sim}\wedge c=0$. Since $c\leq a^{\sim}$, $b^{\prime}\wedge a^{\sim}\wedge
c=b^{\prime}\wedge c$, whence paraorthomodularity of $\mathbf{L}$ yields the
desired result.

($\mathbf{L}_{1}$ is a BZ-lattice). Let $b,c\in\left[  0,a^{\sim}\right]  $.
Clearly, $b\wedge b^{\sim1}=0$. Further, $b^{\sim1\sim1}=\left(  b^{\sim
}\wedge a^{\sim}\right)  ^{\sim}\wedge a^{\sim}$. Since $b^{\sim}$ and
$a^{\sim}$ are sharp elements, $\left(  b^{\sim}\wedge a^{\sim}\right)
^{\sim}=\left(  b^{\sim}\wedge a^{\sim}\right)  ^{\prime}=\Diamond
b\vee\Diamond a$, hence $b^{\sim1\sim1}=\left(  b^{\sim}\wedge a^{\sim
}\right)  ^{\sim}\wedge a^{\sim}=\left(  b^{\sim}\wedge a^{\sim}\right)
^{^{\prime}}\wedge a^{\sim}=b^{\sim1^{\prime}1}$. By Lemma \ref{arismo}.(iii)
$b^{\sim1\sim1}=\left(  \Diamond b\vee\Diamond a\right)  \wedge a^{\sim
}=\Diamond b\wedge a^{\sim}\geq b$. Finally, it is easily seen that if $b\leq
c$, then $c^{\sim1}\leq b^{\sim1}$.

($\mathbf{L}_{1}$ is a BZ*-lattice). Let $b\in\left[  0,a^{\sim}\right]  $.
Then, by Lemma \ref{arismo}.(iii) and WSDM, we have that:%
\[%
\begin{array}
[c]{ll}%
\left(  b\wedge b^{\prime1}\right)  ^{\sim1} & =\left(  b\wedge b^{\prime
}\wedge a^{\sim}\right)  ^{\sim}\wedge a^{\sim}\\
& =\left(  b\wedge b^{\prime}\right)  ^{\sim}\wedge a^{\sim}\\
& =\left(  b^{\sim}\vee\square b\right)  \wedge a^{\sim}\\
& =\left(  b^{\sim}\wedge a^{\sim}\right)  \vee\left(  \square b\wedge
a^{\sim}\right) \\
& =\left(  b^{\sim}\wedge a^{\sim}\right)  \vee\left(  \left(  b^{\prime
}\wedge a^{\sim}\right)  ^{\sim}\wedge a^{\sim}\right) \\
& =b^{\sim1}\vee b^{\prime1\sim1}\text{.}%
\end{array}
\]

($\mathbf{L}_{1}\in V\left(  \mathbb{AOL}\right)  $). By way of example, we
check that the reformulation of AOL2 in terms of the new operations
$^{\sim1\text{ }}$ and $^{\prime1}$ is satisfied in any antiortholattice
$\mathbf{M}$. Thus, for $a,b,c\in\mathbf{M}$, consider the element%
\[
t^{\mathbf{M}}\left(  a,b,c\right)  =\left(  b\wedge c^{\sim}\wedge a^{\sim
}\right)  \vee\left(  b\wedge\left(  c^{\sim}\wedge a^{\sim}\right)  ^{\sim
}\wedge a^{\sim}\right)  \text{.}%
\]

If $c>0$, then $t^{\mathbf{M}}\left(  a,b,c\right)  =0\vee\left(
b\wedge1\wedge a^{\sim}\right)  =b\wedge a^{\sim}$. If $c=0$, then%
\[
t^{\mathbf{M}}\left(  a,b,c\right)  =\left(  b\wedge a^{\sim}\right)
\vee\left(  b\wedge\Diamond a\wedge a^{\sim}\right)  =\left(  b\wedge a^{\sim
}\right)  \vee0=b\wedge a^{\sim}.
\]

Since $\mathbf{L}\in V\left(  \mathbb{AOL}\right)  $, it follows that for all
$a,b,c\in L$, $t^{\mathbf{L}}\left(  a,b,c\right)  =b\wedge a^{\sim}$, which
equals $b$ whenever $b\leq a^{\sim}$, hence $\mathbf{L}_{1}$ satisfies AOL2.
\end{proof}

\begin{theorem}
Let $\mathbf{L}\in V\left(  \mathbb{AOL}\right)  $ and let $a\in L$. Then
$\mathbf{L}\simeq\mathbf{L}_{1}\times\mathbf{L}_{2}$, where $\mathbf{L}%
_{1},\mathbf{L}_{2}$ are defined as in Lemma \ref{nilla}.
\end{theorem}

\begin{proof}
Let $\varphi:L\rightarrow L_{1}\times L_{2}$ be defined, for any $b\in L$, by
$\varphi\left(  b\right)  =\left(  b\wedge a^{\sim},b\wedge\Diamond a\right)
$. We first show that $\varphi$ is a bijection. If $\varphi\left(  b\right)
=\varphi\left(  c\right)  $ for some $b,c\in L$, then $b\wedge a^{\sim
}=c\wedge a^{\sim}$ and $b\wedge\Diamond a=c\wedge\Diamond a$. Thus, by AOL2,
\[
b=\left(  b\wedge a^{\sim}\right)  \vee\left(  b\wedge\Diamond a\right)
=\left(  c\wedge a^{\sim}\right)  \vee\left(  c\wedge\Diamond a\right)
=c\text{.}%
\]

Now, let $\left(  x,y\right)  \in L_{1}\times L_{2}$. Then $x\leq a^{\sim}$
and $y\leq\Diamond a$, whence $y\wedge a^{\sim}\leq\Diamond a\wedge a^{\sim
}=0$. Let us compute $\varphi\left(  x\vee y\right)  $. Using Lemma
\ref{arismo}.(iii), we obtain%
\[%
\begin{array}
[c]{ll}%
\left(  x\vee y\right)  \wedge a^{\sim} & =\left(  x\wedge a^{\sim}\right)
\vee\left(  y\wedge a^{\sim}\right) \\
& =x\vee0=x\text{.}%
\end{array}
\]

Similarly, $\left(  x\vee y\right)  \wedge\Diamond a=y$ and thus $\varphi$ is onto.

Next, we show that $\varphi$ preserves meets and the unary operations (observe
that this is sufficient in virtue of Lemma \ref{nilla}). For meets,%
\[%
\begin{array}
[c]{ll}%
\varphi\left(  b\wedge^{\mathbf{L}}c\right)  & =\left(  b\wedge c\wedge
a^{\sim},b\wedge c\wedge\Diamond a\right) \\
& =\left(  b\wedge a^{\sim},b\wedge\Diamond a\right)  \wedge^{\mathbf{L}%
_{1}\times\mathbf{L}_{2}}\left(  c\wedge a^{\sim},c\wedge\Diamond a\right) \\
& =\varphi\left(  b\right)  \wedge^{\mathbf{L}_{1}\times\mathbf{L}_{2}}%
\varphi\left(  c\right)  \text{.}%
\end{array}
\]

With regards to Kleene complements, resorting again to Lemma \ref{arismo}%
.(iii),
\[%
\begin{array}
[c]{ll}%
\varphi\left(  b^{\prime\mathbf{L}}\right)  & =\left(  b^{\prime}\wedge
a^{\sim},b^{\prime}\wedge\Diamond a\right) \\
& =\left(  \left(  b^{\prime}\vee\Diamond a\right)  \wedge a^{\sim},\left(
b^{\prime}\vee a^{\sim}\right)  \wedge\Diamond a\right) \\
& =\left(  \left(  b\wedge a^{\sim}\right)  ^{\prime}\wedge a^{\sim},\left(
b\wedge\Diamond a\right)  ^{\prime}\wedge\Diamond a\right) \\
& =\left(  \left(  b\wedge a^{\sim}\right)  ^{\prime1},\left(  b\wedge\Diamond
a\right)  ^{\prime2}\right) \\
& =\varphi\left(  b\right)  ^{\prime\mathbf{L}_{1}\times\mathbf{L}_{2}%
}\text{.}%
\end{array}
\]

For Brouwer complements the computation is similar. It is safely left to the
reader, who is warned that the WSDM identity $\left(  x\wedge y^{\sim}\right)
^{\sim}\approx x^{\sim}\vee\Diamond y$ will be needed somewhere down the line.
\end{proof}

\section{Ideal Theory\label{Scipio}}

It is well-known from the theory of orthomodular lattices that $\mathbb{OML}%
$-ideals admit a manageable characterization in terms of lattice ideals
\emph{closed under perspectivity}, for short \emph{p-ideals}: in other words,
in terms of lattice ideals $I$ of an orthomodular lattice $\mathbf{L}$ such
that, whenever $a\in I$, then also $b\Cap a=b\wedge\left(  b^{\prime}\vee
a\right)  \in I$ for all $b\in L$ \cite[Prop. 4.7]{BH}. The aim of this
section is to generalise this idea within the expanded language of
$\mathbb{PBZL}^{\mathbb{\ast}}$. Unfortunately, in the general case these
p-ideals do not even coincide with $0$-classes of congruences --- and, a
fortiori, no isomorphism result between the lattices of p-ideals and of
congruences can be attained. The situation improves if we restrict ourselves
to the subvariety $\mathbb{SDM}$ of $\mathbb{PBZL}^{\mathbb{\ast}}$,
axiomatised relative to $\mathbb{PBZL}^{\mathbb{\ast}}$ by the Strong De
Morgan identity. In fact, in any $\mathbf{L}\in\mathbb{SDM}$ p-ideals coincide
with Ursini $\mathbb{SDM}$-ideals, hence with $0$-classes of congruences,
given the fact that $\mathbb{PBZL}^{\mathbb{\ast}}$ (and thus, all the more
so, $\mathbb{SDM}$) is a $0$-subtractive variety.

\subsection{Ideals in PBZ$^{\ast}$ --lattices}

We start by defining the notion of a p-ideal for generic PBZ$^{\ast}$ --lattices.

\begin{definition}
\label{cagetti}Let $\mathbf{L}$ be a PBZ$^{\ast}$ --lattice. $I\subseteq L$ is
a \emph{p-ideal} iff it is a lattice ideal of $\mathbf{L}$ s.t. if $a\in I$,
then $\Diamond b\Cap\Diamond a=\Diamond b\wedge\left(  b^{\sim}\vee\Diamond
a\right)  \in I$ for all $b\in L$.
\end{definition}

\begin{lemma}
\label{nicotra}Let $\mathbf{L}$ be a PBZ$^{\ast}$ --lattice, let $I$ be a
p-ideal of $\mathbf{L}$, and let $a,b\in L$. Then: (i) if $a\in I$, then
$\Diamond a\in I$; (ii) $\Diamond a\Cap\Diamond b\in I$ iff $\Diamond
b\Cap\Diamond a\in I$.
\end{lemma}

\begin{proof}
(i) Let $b=1$ in Definition \ref{cagetti}.

(ii) Suppose that $\Diamond b\wedge\left(  b^{\sim}\vee\Diamond a\right)  \in
I$. Then, since $I$ is a p-ideal,%
\[
\Diamond a\wedge\left(  a^{\sim}\vee\Diamond\left(  \Diamond b\wedge\left(
b^{\sim}\vee\Diamond a\right)  \right)  \right)  =\Diamond a\wedge\left(
a^{\sim}\vee\left(  \Diamond b\wedge\left(  b^{\sim}\vee\Diamond a\right)
\right)  \right)  \in I\text{.}%
\]

By Theorem \ref{fuli} applied to $\mathbf{S}_{K}\left(  \mathbf{L}\right)  $,
however,%
\[
\Diamond a\wedge\left(  a^{\sim}\vee\left(  \Diamond b\wedge\left(  b^{\sim
}\vee\Diamond a\right)  \right)  \right)  =\Diamond a\wedge\left(  a^{\sim
}\vee\Diamond b\right)  \text{,}%
\]
whence our conclusion.
\end{proof}

We now aim at defining a customary notion of equivalence between elements in a
PBZ$^{\ast}$ --lattice modulo a given p-ideal. Note that $\left\{  0\right\}
$ is a p-ideal; thus,

\begin{definition}
Let $\mathbf{L}$ be a PBZ$^{\ast}$ --lattice, and let $I$ be a p-ideal of
$\mathbf{L}$. The elements $a,b\in L$ are said to be $I$\emph{-modally
equivalent} iff%
\[
\left(  \Diamond a\right)  ^{\sim}\Cap\Diamond b,\left(  \Diamond b\right)
^{\sim}\Cap\Diamond a,\left(  \square a\right)  ^{\sim}\Cap\square b,\left(
\square b\right)  ^{\sim}\Cap\square a\in I\text{.}%
\]

The elements $a,b\in L$ are said to be \emph{modally equivalent} iff they are
$\left\{  0\right\}  $-modally equivalent, namely iff $\Diamond a=\Diamond b$
and $\square a=\square b$.
\end{definition}

\begin{definition}
Let $\mathbf{L}$ be a PBZ$^{\ast}$ --lattice, and let $I$ be a p-ideal of
$\mathbf{L}$. We define the following binary relation on $L$:
\[
\rho\left(  I\right)  =\{\left(  a,b\right)  \in L:\left(  \Diamond a\right)
^{\sim}\Cap\Diamond b,\left(  \Diamond b\right)  ^{\sim}\Cap\Diamond a,\left(
\square a\right)  ^{\sim}\Cap\square b,\left(  \square b\right)  ^{\sim}%
\Cap\square a\in I\}\text{.}%
\]

\end{definition}

Thus, $\left(  a,b\right)  \in\rho\left(  I\right)  $ iff $a$ and $b$ are
$I$-modally equivalent, and, in particular, $\left(  a,b\right)  \in
\rho\left(  \{0\}\right)  $ iff $a$ and $b$ are modally equivalent. The
relation $\rho\left(  I\right)  $ also admits a less cumbersome description:

\begin{theorem}
\label{balanzone}Let $\mathbf{L}$ be a PBZ$^{\ast}$ --lattice, and let $I$ be
a p-ideal of $\mathbf{L}$. For $a,b\in L$ the following conditions are equivalent:

\begin{enumerate}
\item $\left(  \Diamond a\vee\Diamond b\right)  \wedge\left(  a^{\sim}\vee
b^{\sim}\right)  ,\left(  \left(  \square a\right)  ^{\sim}\vee\left(  \square
b\right)  ^{\sim}\right)  \wedge\left(  \square a\vee\square b\right)  \in I$;

\item there exist $s,t\in I$ such that $\Diamond a\vee s=\Diamond b\vee s$ and
$\square a\vee t=\square b\vee t$;

\item $\left(  a,b\right)  \in\rho\left(  I\right)  $.
\end{enumerate}
\end{theorem}

\begin{proof}
(i) implies (ii). Let $s=\left(  \Diamond a\vee\Diamond b\right)
\wedge\left(  a^{\sim}\vee b^{\sim}\right)  $ and $t=\left(  \left(  \square
a\right)  ^{\sim}\vee\left(  \square b\right)  ^{\sim}\right)  \wedge\left(
\square a\vee\square b\right)  $. Then:%
\[%
\begin{array}
[c]{lll}%
\Diamond a\vee s & =\Diamond a\vee\left(  \left(  \Diamond a\vee\Diamond
b\right)  \wedge\left(  a^{\sim}\vee b^{\sim}\right)  \right)  & \\
& =\Diamond a\vee\Diamond b & \text{(Thm. \ref{fuli} in }\mathbf{S}_{K}\left(
\mathbf{L}\right)  \text{)}\\
& =\Diamond b\vee s. &
\end{array}
\]

Similarly, $\square a\vee t=\square b\vee t$.

(ii) implies (iii). Suppose that there exist $s,t\in I$ such that $\Diamond
a\vee s=\Diamond b\vee s$ and $\square a\vee t=\square b\vee t$. First, we
prove $\left(  \Diamond b\right)  ^{\sim}\Cap\Diamond a\in I$.
\[%
\begin{array}
[c]{ll}%
\left(  \Diamond b\right)  ^{\sim}\Cap\Diamond a & =b^{\sim}\wedge\left(
\Diamond b\vee\Diamond a\right) \\
& \leq b^{\sim}\wedge\left(  \Diamond s\vee\Diamond b\vee\Diamond a\right) \\
& =b^{\sim}\wedge\left(  \Diamond\left(  s\vee\Diamond a\right)  \vee\Diamond
b\right) \\
& =b^{\sim}\wedge\left(  \Diamond\left(  s\vee\Diamond b\right)  \vee\Diamond
b\right) \\
& =b^{\sim}\wedge\left(  \Diamond b\vee\Diamond s\right) \\
& =\left(  \Diamond b\right)  ^{\sim}\Cap\Diamond s\text{.}%
\end{array}
\]

Since $I$ is a p-ideal of $\mathbf{L}$ and $s\in I$, we have that $\left(
\Diamond b\right)  ^{\sim}\Cap\Diamond s\in I$ and therefore $\left(  \Diamond
b\right)  ^{\sim}\Cap\Diamond a\in I$. The remaining conditions are proved
similarly and thus $\left(  a,b\right)  \in\rho\left(  I\right)  $.

(iii) implies (i). Suppose $\left(  a,b\right)  \in\rho\left(  I\right)  $; we
prove that $\left(  \Diamond a\vee\Diamond b\right)  \wedge\left(  a^{\sim
}\vee b^{\sim}\right)  \in I$. By assumption,
\[
\left(  a^{\sim}\wedge\left(  \Diamond a\vee\Diamond b\right)  \right)
\vee\left(  b^{\sim}\wedge\left(  \Diamond b\vee\Diamond a\right)  \right)
=\left(  \left(  \Diamond a\right)  ^{\sim}\Cap\Diamond b\right)  \vee\left(
\left(  \Diamond b\right)  ^{\sim}\Cap\Diamond a\right)  \in I.
\]

By Theorem \ref{fuli} applied to $\mathbf{S}_{K}\left(  \mathbf{L}\right)  $,
however,%
\[
\left(  a^{\sim}\wedge\left(  \Diamond a\vee\Diamond b\right)  \right)
\vee\left(  b^{\sim}\wedge\left(  \Diamond b\vee\Diamond a\right)  \right)
=\left(  \Diamond a\vee\Diamond b\right)  \wedge\left(  a^{\sim}\vee b^{\sim
}\right)  ,
\]
whence our claim follows. A similar proof establishes the other claim.
\end{proof}

\begin{theorem}
\label{nillapenna} Let $\mathbf{L}$ be a PBZ$^{\ast}$ --lattice, and let $I$
be a p-ideal of $\mathbf{L}$. Then $\rho\left(  I\right)  $ is an equivalence
relation on $L$ that preserves the operations $^{\prime}$ and $^{\sim}$.
\end{theorem}

\begin{proof}
Since, for all $a\in L$, $\left(  \Diamond a\right)  ^{\sim}\Cap\Diamond
a,\left(  \square a\right)  ^{\sim}\Cap\square a=0\in I$, $\rho\left(
I\right)  $ is reflexive. Symmetry is trivial. For transitivity, suppose
$\left(  a,b\right)  ,\left(  b,c\right)  \in\rho\left(  I\right)  $. By
Theorem \ref{balanzone}, there exist:

\begin{itemize}
\item $s_{1},t_{1}\in I$ such that $\Diamond a\vee s_{1}=\Diamond b\vee s_{1}$
and $\square a\vee t_{1}=\square b\vee t_{1}$;

\item $s_{2},t_{2}\in I$ such that $\Diamond b\vee s_{2}=\Diamond c\vee s_{2}$
and $\square b\vee t_{2}=\square c\vee t_{2}$.
\end{itemize}

Thus $s_{1}\vee s_{2}\in I$ and $\Diamond a\vee s_{1}\vee s_{2}=\Diamond b\vee
s_{1}\vee s_{2}=\Diamond c\vee s_{1}\vee s_{2}$, and similarly for the other
condition, whence by Theorem \ref{balanzone} again, $\left(  a,c\right)
\in\rho\left(  I\right)  $. The unary operations are clearly preserved.
\end{proof}

Although $\rho\left(  I\right)  $ is always an equivalence relation, it need
not always be a congruence, as the next example shows.

\begin{example}
\label{cogotti}Consider the distributive antiortholattice whose lattice reduct
is the ordinal sum of $\mathbf{D}_{2}^{2}$ with itself, with atoms $a,b$ and
the fixpoint $c=c^{\prime}$. Observe that $\left(  a,c\right)  \in\rho\left(
\{0\}\right)  $, because $\Diamond a=\Diamond c=1$ and $\square a=\square
c=0$. However, $\Diamond\left(  a\wedge b\right)  =\Diamond0=0$ and
$\Diamond\left(  c\wedge b\right)  =\Diamond b=1$, whence $\rho\left(
\{0\}\right)  $ does not preserve meets.
\end{example}

Our next goal is to tweak the notion of p-ideal in such a way that its
associated equivalence is necessarily a congruence.

\begin{definition}
Let $\mathbf{L}$ be a PBZ$^{\ast}$ --lattice, and let $I$ be a p-ideal of
$\mathbf{L}$. $I$ is a \emph{weak De Morgan ideal} iff for all $a,b\in L$,
whenever $\left(  a,b\right)  \in\rho\left(  I\right)  $, then for all $c\in
L$ it is the case that $\Diamond\left(  a\wedge c\right)  ^{\sim}\Cap
\Diamond\left(  b\wedge c\right)  \in I$.
\end{definition}

\begin{lemma}
\label{miriamagus}Let $\mathbf{L}$ be a PBZ$^{\ast}$ --lattice, and let $I$ be
a p-ideal of $\mathbf{L}$. The following conditions are equivalent:

\begin{enumerate}
\item $\rho\left(  I\right)  $ is a congruence;

\item $I$ is a weak De Morgan ideal.
\end{enumerate}
\end{lemma}

\begin{proof}
(i) implies (ii). Suppose that $\rho\left(  I\right)  $ is a congruence and
let $\left(  a,b\right)  \in\rho\left(  I\right)  $, $c\in L$. Then $\left(
\Diamond\left(  a\wedge c\right)  ,\Diamond\left(  b\wedge c\right)  \right)
\in\rho\left(  I\right)  $. It follows that $(\Diamond(a\wedge c)^{\sim}%
\Cap\Diamond(b\wedge c),0)\in\rho(I)$, which implies $\Diamond\left(  a\wedge
c\right)  ^{\sim}\Cap\Diamond\left(  b\wedge c\right)  \in I$.

(ii) implies (i). Let $I$ be a weak De Morgan ideal. By Theorem
\ref{nillapenna}, to attain our conclusion it will suffice to show that
$\rho\left(  I\right)  $ preserves meets. Thus, let $\left(  a,b\right)
\in\rho\left(  I\right)  $ and $c\in L$. In virtue of our assumption,
$\Diamond\left(  a\wedge c\right)  ^{\sim}\Cap\Diamond\left(  b\wedge
c\right)  \in I$ and, taking into account the symmetry of $\rho\left(
I\right)  $, $\Diamond\left(  b\wedge c\right)  ^{\sim}\Cap\Diamond\left(
a\wedge c\right)  \in I$. It remains to show that%
\[
\left(  \square\left(  a\wedge c\right)  \right)  ^{\sim}\Cap\square\left(
b\wedge c\right)  ,\left(  \square\left(  b\wedge c\right)  \right)  ^{\sim
}\Cap\square\left(  a\wedge c\right)  \in I.
\]

However, since $\left(  a,b\right)  \in\rho\left(  I\right)  $ and
$\rho\left(  I\right)  $ preserves the unary operations, $(\square a,\square
b)\in\rho\left(  I\right)  $. Given that $I$ is a weak De Morgan ideal, thus,%
\[%
\begin{array}
[c]{ll}%
\left(  \square\left(  a\wedge c\right)  \right)  ^{\sim}\Cap\square\left(
b\wedge c\right)   & =\Diamond\left(  a^{\prime}\vee c^{\prime}\right)
\Cap\left(  b^{\prime}\vee c^{\prime}\right)  ^{\sim}\\
& =\left(  \square a\wedge\square c\right)  ^{\sim}\Cap\Diamond\left(  \square
b\wedge\square c\right)  \\
& =\Diamond\left(  \square a\wedge\square c\right)  ^{\sim}\Cap\Diamond\left(
\square b\wedge\square c\right)  \in I\text{.}%
\end{array}
\]

Similarly, $\left(  \square\left(  b\wedge c\right)  \right)  ^{\sim}%
\Cap\square\left(  a\wedge c\right)  \in I$.
\end{proof}

\subsection{Ideals in the Strong De Morgan Subvariety}

The subvariety of $\mathbb{PBZL}^{\mathbb{\ast}}$ that is axiomatised relative
to $\mathbb{PBZL}^{\mathbb{\ast}}$ by the Strong De Morgan law SDM, here
labelled $\mathbb{SDM}$, includes $\mathbb{OML}$ and stands out for its smooth
theory of ideals. In fact, we have that:

\begin{lemma}
\label{stellaconte}Let $\mathbf{L}\in\mathbb{SDM}$, and let $I$ be a p-ideal
of $\mathbf{L}$. Then $I$ is a weak De Morgan ideal and therefore $\rho\left(
I\right)  $ is a congruence.
\end{lemma}

\begin{proof}
If $\left(  a,b\right)  \in\rho\left(  I\right)  $, then $\left(  a^{\sim
},b^{\sim}\right)  \in\rho\left(  I\right)  $ by Theorem \ref{nillapenna}.
Thus, Theorem \ref{balanzone} guarantees that there is $s\in I$ such that
$a^{\sim}\vee s=\Diamond a^{\sim}\vee s=\Diamond b^{\sim}\vee s=b^{\sim}\vee
s$. Then, for an arbitrary $c\in L$, $a^{\sim}\vee c^{\sim}\vee s=b^{\sim}\vee
c^{\sim}\vee s$. Applying SDM, we have that $\square\left(  \left(  a\wedge
c\right)  ^{\sim}\right)  \vee s=\left(  a\wedge c\right)  ^{\sim}\vee
s=\left(  b\wedge c\right)  ^{\sim}\vee s=\square\left(  \left(  b\wedge
c\right)  ^{\sim}\right)  \vee s$. A further recourse to Theorem
\ref{balanzone} yields $\left(  \left(  a\wedge c\right)  ^{\sim},\left(
b\wedge c\right)  ^{\sim}\right)  \in\rho\left(  I\right)  $, whence $\left(
\Diamond\left(  a\wedge c\right)  ,\Diamond\left(  b\wedge c\right)  \right)
\in\rho\left(  I\right)  $, which implies, in particular, that $I$ is weak De
Morgan. Lemma \ref{miriamagus} takes care of the remaining claim.
\end{proof}

We are now in a position to prove that within the boundaries of this
subvariety, p-ideals coincide with ideals in the sense of Ursini.

\begin{theorem}
\label{cattanei}If $\mathbf{L}\in\mathbb{SDM}$, then the class of p-ideals of
$\mathbf{L}$ coincides with $\mathcal{I}_{\mathbb{SDM}}\left(  \mathbf{L}%
\right)  $.
\end{theorem}

\begin{proof}
Let $I\in\mathcal{I}_{\mathbb{SDM}}\left(  \mathbf{L}\right)  $, whence by
Theorem \ref{funziona}.(i) $I=0/\theta$ for some $\theta\in\mathrm{Con}%
_{\mathbb{BZL}}\left(  \mathbf{L}\right)  $. Clearly, $I$ is a lattice ideal
of $\mathbf{L}$. Furthermore, if $a\in I$, then $\left(  \Diamond a,0\right)
\in\theta$ and then $\Diamond a\in I$. What remains to show is that, for an
arbitrary $b\in L$, $\Diamond b\Cap\Diamond a\in I$. Since $\Diamond a\in
I=0/\theta$, $\left(  \Diamond b\Cap\Diamond a,0\right)  =\left(  \Diamond
b\Cap\Diamond a,\Diamond b\Cap0\right)  \in\theta$, which means $\Diamond
b\Cap\Diamond a\in I$. Conversely, it will be enough to prove that if $I$ is a
p-ideal of $\mathbf{L}$, then $I=0/\rho\left(  I\right)  $. However, by
Theorem \ref{balanzone},
\begin{align*}
0/\rho\left(  I\right)   &  =\left\{  a\in L:\left(  a,0\right)  \in
\rho\left(  I\right)  \right\} \\
&  =\left\{  a\in L:\Diamond a\leq s,\square a\leq t\text{ for some }s,t\in
I\right\}  \text{.}%
\end{align*}
If $a\in I$, then choose $s=t=\Diamond a\in I$ to obtain $a\in0/\rho\left(
I\right)  $. If $a\in0/\rho\left(  I\right)  $, then there is $s\in I$ such
that $a\leq\Diamond a\leq s$, whence $a\in I$.
\end{proof}

Observe that, by Lemma \ref{stellaconte} and Theorem \ref{cattanei}, whenever
$\mathbf{L}\in\mathbb{SDM}$, all members of $\mathcal{I}_{\mathbb{SDM}}\left(
\mathbf{L}\right)  $ are weak De Morgan ideals.

\begin{theorem}
\label{orsucci} Let $\mathbf{L}\in\mathbb{SDM}$, and let $I\in\mathcal{I}%
_{\mathbb{SDM}}\left(  \mathbf{L}\right)  $. Then $\rho\left(  I\right)
=I^{\varepsilon}$.
\end{theorem}

\begin{proof}
By the proof of Theorem \ref{cattanei} $0/\rho\left(  I\right)  =I$, whence
$\rho\left(  I\right)  \subseteq I^{\varepsilon}$. For the converse
inequality, suppose $\left(  a,b\right)  \in I^{\varepsilon}$. Since
$I^{\varepsilon}$ is a congruence,%
\[
\left(  \left(  \Diamond a\vee\Diamond b\right)  \wedge\left(  a^{\sim}\vee
b^{\sim}\right)  ,0\right)  ,\left(  \left(  \square a\vee\square b\right)
\wedge\left(  \left(  \square a\right)  ^{\sim}\vee\left(  \square b\right)
^{\sim}\right)  ,0\right)  \in I^{\varepsilon}\text{.}%
\]
So $\left(  \Diamond a\vee\Diamond b\right)  \wedge\left(  a^{\sim}\vee
b^{\sim}\right)  \in0/I^{\varepsilon}=I$ and $\left(  \square a\vee\square
b\right)  \wedge\left(  \left(  \square a\right)  ^{\sim}\vee\left(  \square
b\right)  ^{\sim}\right)  \in0/I^{\varepsilon}=I$. By Theorem \ref{balanzone},
this means that $\left(  a,b\right)  \in\rho\left(  I\right)  $.
\end{proof}

\begin{corollary}
\label{micaela} $\mathbb{SDM}$ is finitely congruential.
\end{corollary}

\begin{proof}
We have to find a finite set of terms $\{d_{i}\left(  x,y\right)  \}_{i\leq
n}$ that witnesses finite congruentiality according to Definition
\ref{granata}. Thus, let%
\[%
\begin{array}
[c]{cc}%
d_{1}\left(  x,y\right)  =\left(  \Diamond x\right)  ^{\sim}\Cap\Diamond y, &
d_{2}\left(  x,y\right)  =\left(  \Diamond y\right)  ^{\sim}\Cap\Diamond x,\\
d_{3}\left(  x,y\right)  =\left(  \square x\right)  ^{\sim}\Cap\square y, &
d_{4}\left(  x,y\right)  =\left(  \square y\right)  ^{\sim}\Cap\square x.
\end{array}
\]

If $\mathbf{L}\in\mathbb{SDM}$ and $I\in\mathcal{I}_{\mathbb{SDM}}\left(
\mathbf{L}\right)  $, then by Theorems \ref{orsucci} and \ref{balanzone}
$\rho\left(  I\right)  =I^{\varepsilon}$. As a result, $\left(  a,b\right)
\in I^{\varepsilon}=\rho\left(  I\right)  $ iff $d_{i}^{\mathbf{A}}\left(
a,b\right)  \in I\text{, for all }i\leq4$.
\end{proof}

\begin{theorem}
The $0$-assertional logic of $\mathbb{PBZL}^{\mathbb{\ast}}$ is not equivalential.
\end{theorem}

\begin{proof}
Consider again the antiortholattice of Example \ref{cogotti}. Being simple,
this antiortholattice belongs to $\mathbb{PBZL}_{\varepsilon}^{\mathbb{\ast}}%
$. Moreover, the set $\left\{  0,a,a^{\prime},1\right\}  $ is a subuniverse of
such, isomorphic to $\mathbf{D}_{4}$, and its middle congruence, that
collapses only $a$ and $a^{\prime}$, is a nonzero pseudo-identical congruence.
Our claim follows then from Theorem \ref{mistretta}.
\end{proof}

\begin{lemma}
\label{moriconi}Let $\mathbf{L}\in\mathbb{SDM}$. The following are equivalent:

\begin{enumerate}
\item $\rho\left(  \{0\}\right)  =\Delta_{\mathbf{L}}$.

\item $\mathbf{L}$ satisfies the quasi-identity $\square x\leq\square
y\,\&\,\Diamond x\leq\Diamond y\Rightarrow x\leq y$.

\item $\mathbf{L}$ satisfies the identity SK.
\end{enumerate}
\end{lemma}

\begin{proof}
(i) implies (ii). $\rho\left(  \{0\}\right)  =\Delta_{\mathbf{L}}$ means that
modally equivalent elements of $L$ are identical. Now, let $\square
a\leq\square b$\ and $\Diamond a\leq\Diamond b$. By SDM, this implies that
$a\wedge b$ and $a$ are modally equivalent, whence $a\leq b$.

(ii) implies (iii). Using SDM, we have that for all $a,b\in L$,%
\begin{align*}
\square\left(  a\wedge\Diamond b\right)   &  =\square a\wedge\Diamond
b\leq\square a\vee\square b=\square\left(  \square a\vee b\right)  \text{;}\\
\Diamond\left(  a\wedge\Diamond b\right)   &  =\Diamond a\wedge\Diamond
b\leq\square a\vee\Diamond b=\Diamond\left(  \square a\vee b\right)  \text{.}%
\end{align*}

By our assumption, then, $a\wedge\Diamond b\leq\square a\vee b$.

(iii) implies (i). Suppose $\square a=\square b$ and $\Diamond a=\Diamond b$.
Then%
\[
a=a\wedge\Diamond a=a\wedge\Diamond b\leq\square a\vee b=\square b\vee b=b.
\]

Similarly, $b\leq a$, whence our conclusion.
\end{proof}

Let us call $\mathbb{SK}$ the subvariety of $\mathbb{SDM}$ that is axiomatised
relative to $\mathbb{SDM}$ by the identity SK.

\begin{theorem}
\label{melis} The $0$-assertional logic of $\mathbb{SDM}$ is strongly
algebraisable with equivalent variety semantics $\mathbb{SDM}_{\varepsilon
}=\mathbb{SK}$.
\end{theorem}

\begin{proof}
By Theorem \ref{mistrettone}, it suffices to establish that $\mathbb{SDM}%
_{\varepsilon}$ is a variety, which would follow if we were to show that
$\mathbb{SDM}_{\varepsilon}=\mathbb{SK}$. By Theorem \ref{orsucci}, whenever
$\mathbf{L}$ belongs to $\mathbb{SDM}$, $\left\{  0\right\}  ^{\varepsilon
}=\rho\left(  \{0\}\right)  $. Thus $\mathbf{L}\in\mathbb{SDM}_{\varepsilon}$
iff $\rho\left(  \{0\}\right)  =\Delta_{\mathbf{L}}$, and by Lemma
\ref{moriconi}, this happens exactly when $\mathbf{L}\in\mathbb{SK}$.
\end{proof}

By Theorems \ref{crnjar} and \ref{cattanei}, in any member $\mathbf{L}$ of
$\mathbb{SDM}$ the Ursini ideals of $\mathbf{L}$ coincide with its p-ideals
and with the deductive filters on $\mathbf{L}$ of the $0$-assertional logic of
$\mathbb{SDM}$. By {\cite[Thm. 3.58]{Font}, therefore, we obtain:}

\begin{corollary}
Let $\mathbf{L}\in\mathbb{SDM}$. Then the lattice of p-ideals of $\mathbf{L}$
is isomorphic to the lattice of all congruences $\theta$ on $\mathbf{L}$ such
that $\mathbf{L}/\theta\in\mathbb{SK}$.
\end{corollary}

Observe that, although SK implies SDM in the context of $V\left(
\mathbb{AOL}\right)  $ {\cite[Lm. 3.8]{PBZ2}} this is not the case in the more
general context of $\mathbb{PBZL}^{\mathbb{\ast}}$. In fact, consider the
PBZ$^{\ast}$ --lattice $\mathbf{L}$ whose lattice reduct is the $5$-element
modular and non-distributive lattice $\mathbf{M}_{3}$ with atoms $a,a^{\prime
},b$, where $b=b^{\prime}$ and $b^{\sim}=0$. Then $\mathbf{L}$ satisfies SK
but fails SDM - actually, it fails even WSDM because $a\in S_{K}\left(
\mathbf{L}\right)  $ and $\left(  a\wedge b\right)  ^{\sim}=1$ but $a^{\sim
}\vee b^{\sim}=a^{\prime}$. For future reference, we make a note of the fact
that $\mathbf{L}$ satisfies J2.

\subsection{Ideals in $V\left(  \mathbb{AOL}\right)  $}

Another subvariety of $\mathbb{PBZL}^{\mathbb{\ast}}$ where our description of
ideals can be considerably simplified is the variety $V\left(  \mathbb{AOL}%
\right)  $ generated by all antiortholattices. Bignall and Spinks first
observed that the variety of distributive BZ-lattices is a binary
discriminator variety {\cite{BS}}. We extend their observation by noticing
that $V\left(  \mathbb{AOL}\right)  $ is itself a binary discriminator variety.

\begin{proposition}
\label{bischero}$V\left(  \mathbb{AOL}\right)  $ is a $0$-binary discriminator variety.
\end{proposition}

\begin{proof}
Referring to Definition \ref{zanzibar} for notation and terminology, let
$b_{0}^{L}\left(  x,y\right)  =x\wedge y^{\sim}$. Then for any
antiortholattice $\mathbf{L}$ and, for any $a,c\in L$, $b_{0}^{L}\left(
a,0\right)  =a\wedge0^{\sim}=a\wedge1=a$, while, if $c>0$, then $b_{0}%
^{L}\left(  a,c\right)  =a\wedge c^{\sim}=a\wedge0=0$.
\end{proof}

Among the consequences of this remark we have a very slender description of
$V\left(  \mathbb{AOL}\right)  $-ideals. In fact, recall from {\cite{BS} that
if }$\mathbf{A}$ is an algebra in a $0$-binary discriminator variety
$\mathbb{V}$ and $b_{0}^{{}}\left(  x,y\right)  $ is the term witnessing this
property for $\mathbb{V}$, $I\subseteq A$ is a $\mathbb{V}$-ideal of
$\mathbf{A}$ exactly when, for any $a,c\in A$, if $c\in I$ and $b_{0}%
^{\mathbf{A}}\left(  a,c\right)  \in I$, then $a\in I$. Therefore:

\begin{proposition}
Let $\mathbf{L}\in V\left(  \mathbb{AOL}\right)  $. For a lattice ideal
$I\subseteq L$ the following are equivalent:

\begin{enumerate}
\item $I$ is a $V\left(  \mathbb{AOL}\right)  $-ideal.

\item For any $a,b\in L$, if $b\in I$ and $a\wedge b^{\sim}\in I$, then $a\in
I$.

\item $I$ is closed w.r.t. all interpretations in $\mathbf{L}$ of the
$V\left(  \mathbb{AOL}\right)  $-ideal term (in $y,z$):
\[
u\left(  x,y,z\right)  =x\wedge\Diamond\left(  y\vee z\right)  .
\]

\end{enumerate}
\end{proposition}

\begin{proof}
The equivalence of (i) and (ii) follows from the remarks immediately preceding
this lemma. Suppose now that (ii) holds, and that $a,b\in I$. Then $a\vee b\in
I$. On the other hand, for any $c$ in $L$,%
\[
0=c\wedge\Diamond\left(  a\vee b\right)  \wedge\left(  a\vee b\right)  ^{\sim
}\in I\text{,}%
\]
whence by our hypothesis $c\wedge\Diamond\left(  a\vee b\right)  \in I$.
Conversely, under the assumption (iii), let $b\in I$ and $a\in L$ be such that
$a\wedge b^{\sim}\in I$. Then, using WSDM and Lemma \ref{arismo} several
times,%
\begin{align*}
u^{\mathbf{A}}\left(  a,b,a\wedge b^{\sim}\right)   &  =a\wedge\Diamond\left(
b\vee\left(  a\wedge b^{\sim}\right)  \right) \\
&  =a\wedge\left(  \Diamond b\vee\Diamond\left(  a\wedge b^{\sim}\right)
\right) \\
&  =a\wedge\left(  \Diamond b\vee\left(  \Diamond a\wedge b^{\sim}\right)
\right) \\
&  =a\wedge\left(  \Diamond b\vee\Diamond a\right) \\
&  =\left(  a\wedge\Diamond b\right)  \vee a=a,
\end{align*}
hence $a\in I$.
\end{proof}

\section{Axiomatic Bases for Some Subvarieties\label{antis}}

The goal of this final section is to simplify the axiomatisation of $V\left(
\mathbb{AOL}\right)  $ given in \cite{GLP1+} and to solve a problem (here
called the Join Problem) posed in \cite{PBZ2}, where it was observed that the
varietal join $\mathbb{OML\vee}V\left(  \mathbb{AOL}\right)  $ in the lattice
of subvarieties of $\mathbb{PBZL}^{\mathbb{\ast}}$ was strictly included in
$\mathbb{PBZL}^{\mathbb{\ast}}$, but no axiomatic basis for such a join was given.

\subsection{A Streamlined Axiomatisation for $V\left(  \mathbb{AOL}\right)  $}

In Theorem \ref{caniggia}.(i) we recalled that an equational basis for
$V\left(  \mathbb{AOL}\right)  $ relative to $\mathbb{PBZL}^{\ast}$ is given
by the identities AOL1-AOL3, here reproduced for the reader's convenience:%
\begin{align*}
\text{(AOL1) }  &  \left(  x^{\sim}\vee y^{\sim}\right)  \wedge\left(
\Diamond x\vee z^{\sim}\right)  \approx\left(  \left(  x^{\sim}\vee y\right)
\wedge\left(  \Diamond x\vee z\right)  \right)  ^{\sim}\text{;}\\
\text{(AOL2) }  &  x\approx\left(  x\wedge y^{\sim}\right)  \vee\left(
x\wedge\Diamond y\right)  \text{;}\\
\text{(AOL3) }  &  x\approx\left(  x\vee y^{\sim}\right)  \wedge\left(
x\vee\Diamond y\right)  \text{.}%
\end{align*}

The aim of this subsection is showing that AOL2 suffices to derive the
remaining two axioms. For a start, we notice that Lemma \ref{arismo} does not
depend on AOL1, whence it holds for any subvariety of $\mathbb{PBZL}^{\ast}$
that satisfies AOL2 and AOL3.

\begin{lemma}
\label{minchione}Let $\mathbf{L}$ be a member of  $\mathbb{PBZL}^{\ast}$ that
satisfies AOL2 and AOL3. Then, for any $a,b,c\in L$: (i) $a\wedge b\leq\left(
a\wedge c^{\sim}\right)  \vee\left(  b\wedge\Diamond c\right)  $; (ii)
$\left(  a\wedge\Diamond b\right)  ^{\sim}\vee\Diamond b=1$; (iii) $\Diamond
a\leq\left(  b\wedge a^{\sim}\right)  ^{\sim}$.
\end{lemma}

\begin{proof}
(i) In fact, using Lemma \ref{arismo},
\[
a\wedge b\leq\left(  a\vee c^{\sim}\right)  \wedge\left(  b\vee\Diamond
c\right)  \wedge\left(  a\vee b\right)  =\left(  a\wedge c^{\sim}\right)
\vee\left(  b\wedge\Diamond c\right)  \text{.}%
\]

(ii) Since $a\wedge\Diamond b\leq\Diamond b$, it follows that $b^{\sim}%
\leq\left(  a\wedge\Diamond b\right)  ^{\sim}$, whence $1=b^{\sim}\vee\Diamond
b\leq\left(  a\wedge\Diamond b\right)  ^{\sim}\vee\Diamond b$.

(iii) Since $b\wedge a^{\sim}\leq a^{\sim}$, our conclusion follows.
\end{proof}

\begin{lemma}
\label{buacciolo}Let $\mathbf{L}$ be a member of $\mathbb{PBZL}^{\ast}$ that
satisfies AOL2 and AOL3. Then $\mathbf{L}$ satisfies AOL1 iff it satisfies WSDM.
\end{lemma}

\begin{proof}
From left to right, WSDM can be obtained by taking $y=0$ and applying Lemma
\ref{arismo}.(iii). Conversely, let $\mathbf{L}$ satisfy WSDM, and let
$a,b,c\in L$. Then:%
\[%
\begin{array}
[c]{lll}%
\left(  \left(  a^{\sim}\vee b\right)  \wedge\left(  \Diamond a\vee c\right)
\right)  ^{\sim} & =\left(  \left(  a^{\sim}\wedge c\right)  \vee\left(
\Diamond a\wedge b\right)  \vee\left(  b\wedge c\right)  \right)  ^{\sim} &
\text{Lm. \ref{arismo}}\\
& =\left(  \left(  a^{\sim}\wedge c\right)  \vee\left(  \Diamond a\wedge
b\right)  \right)  ^{\sim} & \text{Lm. \ref{minchione}.(i)}\\
& =\left(  a^{\sim}\wedge c\right)  ^{\sim}\wedge\left(  \Diamond a\wedge
b\right)  ^{\sim} & \text{Lm. \ref{basics}.(iii)}\\
& =\text{ }\left(  a^{\sim}\vee b^{\sim}\right)  \wedge\left(  \Diamond a\vee
c^{\sim}\right)  & \text{WSDM}%
\end{array}
\]

\end{proof}

\begin{theorem}
\label{bumbum}An equational basis for $V\left(  \mathbb{AOL}\right)  $
relative to $\mathbb{PBZL}^{\ast}$ is given by the single identity AOL2.
\end{theorem}

\begin{proof}
By Theorem \ref{caniggia}.(i), an equational basis for $V\left(
\mathbb{AOL}\right)  $ relative to $\mathbb{PBZL}^{\ast}$ is given by the
identities AOL1-AOL3. To attain our conclusion, taking into account Lemma
\ref{buacciolo}, it will suffice to show that: i) any subvariety of
$\mathbb{PBZL}^{\ast}$ that satisfies AOL2 and AOL3 also satisfies WSDM; ii)
any subvariety of $\mathbb{PBZL}^{\ast}$ that satisfies AOL2 also satisfies
AOL3. We establish these claims in reverse order.

i) Let $\mathbf{L}$ belong to any subvariety of $\mathbb{PBZL}^{\ast}$ that
satisfies AOL2, and let $a,b\in L$. Then $a^{\prime}=\left(  a^{\prime}\wedge
b^{\sim}\right)  \vee\left(  a^{\prime}\wedge\Diamond b\right)  $, whence%
\[
a=\left(  \left(  a^{\prime}\wedge b^{\sim}\right)  \vee\left(  a^{\prime
}\wedge\Diamond b\right)  \right)  ^{\prime}=\left(  a^{\prime}\wedge b^{\sim
}\right)  ^{\prime}\wedge\left(  a^{\prime}\wedge\Diamond b\right)  ^{\prime
}=\left(  a\vee\Diamond b\right)  \wedge\left(  a\vee b^{\sim}\right)
\text{.}%
\]

ii) Let $\mathbf{L}$ belong to any subvariety of $\mathbb{PBZL}^{\ast}$ that
satisfies AOL2 (thus also AOL3, by the previous item), and let $a,b\in L$.
Then:%
\[%
\begin{array}
[c]{lll}%
a^{\sim}\vee\Diamond b & =\left(  \left(  a\wedge b^{\sim}\right)  \vee\left(
a\wedge\Diamond b\right)  \right)  ^{\sim}\vee\Diamond b & \text{AOL2}\\
& =\left(  \left(  a\wedge b^{\sim}\right)  ^{\sim}\wedge\left(
a\wedge\Diamond b\right)  ^{\sim}\right)  \vee\Diamond b & \text{Lm.
\ref{basics}.(iii)}\\
& =\left(  \left(  a\wedge b^{\sim}\right)  ^{\sim}\vee\Diamond b\right)
\wedge\left(  \left(  a\wedge\Diamond b\right)  ^{\sim}\vee\Diamond b\right)
& \text{Lm. \ref{arismo}.(iv)}\\
& =\left(  \left(  a\wedge b^{\sim}\right)  ^{\sim}\vee\Diamond b\right)  &
\text{Lm. \ref{minchione}.(ii)}\\
& =\left(  a\wedge b^{\sim}\right)  ^{\sim} & \text{Lm. \ref{minchione}%
.(iii).}%
\end{array}
\]

\end{proof}

Taking into account Lemma \ref{goldrush} and Theorem \ref{centravanti}, we
have that:

\begin{corollary}
\label{feccia}

\begin{enumerate}
\item $V(\mathbb{AOL})$ is the class of all PBZ$^{\ast}$ --lattices
$\mathbf{L}$ such that $C_{p}(\mathbf{L})=S_{K}(\mathbf{L}).$

\item $V(\mathbb{AOL})$ is the class of all PBZ$^{\ast}$ --lattices
$\mathbf{L}$ that satisfy WSDM and are such that $C_{pbz}(\mathbf{L}%
)=S_{K}(\mathbf{L})$.

\item The class of the directly indecomposable members of $V(\mathbb{AOL})$ is
$\mathbb{AOL}$.
\end{enumerate}
\end{corollary}

\subsection{The Join Problem}

We round off this paper by axiomatising the variety $\mathbb{OML\vee
}V(\mathbb{AOL})$, as well as some of its notable subvarieties. Let:

\begin{itemize}
\item $\mathbb{V}_{1}$ be the variety of PBZ$^{\ast}$ --lattices that is
axiomatised relative to $\mathbb{PBZL}^{\mathbb{\ast}}$ by the identities J2
and WSDM;

\item $\mathbb{V}_{2}$ be the variety of PBZ$^{\ast}$ --lattices that is
axiomatised relative to $\mathbb{PBZL}^{\mathbb{\ast}}$ by the identities J2
and SDM;

\item $\mathbb{V}_{3}$ be the variety of PBZ$^{\ast}$ --lattices that is
axiomatised relative to $\mathbb{PBZL}^{\mathbb{\ast}}$ by the identities J2,
WSDM, and SK.
\end{itemize}

Taking into account the results in \cite{PBZ2} and \cite{GLP1+}, as well as
Theorem \ref{bumbum}, in $V(\mathbb{AOL})$, SK implies SDM, that is AOL2 and
SK imply SDM, while AOL2 and SDM do not imply SK. We observed in Section \ref{Scipio} that J2 and SK do not imply WSDM (all the more so, thus, SDM). The following PBZ$^{\ast}$ --lattice:

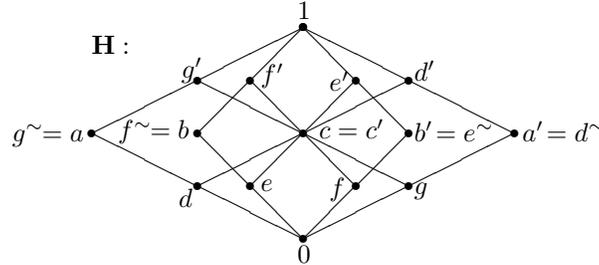
\begin{center}\begin{picture}(40,82)(0,0)
\put(-60,70){${\bf H}:$}
\put(20,0){\circle*{3}}
\put(20,40){\circle*{3}}
\put(20,80){\circle*{3}}
\put(20,80){\circle*{3}}
\put(-60,40){\circle*{3}}
\put(100,40){\circle*{3}}
\put(-20,40){\circle*{3}}
\put(60,40){\circle*{3}}
\put(0,60){\circle*{3}}
\put(-20,60){\circle*{3}}
\put(40,60){\circle*{3}}
\put(60,60){\circle*{3}}
\put(0,20){\circle*{3}}
\put(-20,20){\circle*{3}}
\put(40,20){\circle*{3}}
\put(60,20){\circle*{3}}
\put(-27,12){$d$}
\put(4,18){$e$}
\put(4,59){$f^{\prime }$}
\put(30,56){$e^{\prime }$}
\put(-26,61){$g^{\prime }$}
\put(62,60){$d^{\prime }$}
\put(30,16){$f$}
\put(62,17){$g$}
\put(18,-9){$0$}
\put(18,83){$1$}
\put(-50,38){$f^{\sim }\!\!=b$}
\put(26,38){$c=c^{\prime }$}
\put(62,37){$b^{\prime }\!=e^{\sim }$}
\put(103,37){$a^{\prime }\!=d^{\sim }$}
\put(-90,37){$g^{\sim }\!\!=a$}
\put(20,0){\line(-2,1){80}}
\put(20,0){\line(2,1){80}}
\put(20,0){\line(-1,1){40}}
\put(20,0){\line(1,1){40}}
\put(20,80){\line(-2,-1){80}}
\put(20,80){\line(2,-1){80}}
\put(20,80){\line(-1,-1){40}}
\put(20,80){\line(1,-1){40}}
\put(0,60){\line(1,-1){40}}
\put(-20,60){\line(2,-1){80}}
\put(40,60){\line(-1,-1){40}}
\put(60,60){\line(-2,-1){80}}
\end{picture}\end{center}

\noindent satisfies SK and SDM and fails J2. The $4$-element antiortholattice chain
$\mathbf{D}_{4}$ fails SK but satisfies SDM and AOL2, thus also J2. Therefore,
in the sets of axioms $\left\{  \text{J2, SK, SDM}\right\}  $\ and $\left\{
\text{J2, SK, WSDM}\right\}  $, each axiom is independent from the other two.

Given any {$\mathbf{L}\in\mathbb{PBZL}^{\ast}$}, it will be expedient to
denote by $T\left(  {\mathbf{L}}\right)  $ the set $\left\{  x\in L:x{^{\sim
}=0}\right\}  \cup\left\{  0\right\}  $.

\begin{lemma}
\label{totototo}Let {$\mathbf{L}$} be a PBZ$^{\ast}$ --lattice that satisfies
WSDM and such that $S_{K}(\mathbf{L})\cup T\left(  {\mathbf{L}}\right)  =L$.
Then if $b\in S_{K}(\mathbf{L})$ and $c\notin S_{K}(\mathbf{L})$, it follows
that either $b=1$ or $b\leq c$.
\end{lemma}

\begin{proof}
If $b\wedge c\in S_{K}(\mathbf{L})$, then%
\[
b\wedge c=\Diamond\left(  b\wedge c\right)  =\Diamond b\wedge\Diamond
c=b\wedge1=b\text{,}%
\]
where WSDM can be applied to obtain the second equality because $b\in
S_{K}(\mathbf{L})$, while the third equality follows from the fact that
$c\notin S_{K}(\mathbf{L})$, whence $c{^{\sim}=0}$. On the other hand, if
$b\wedge c\notin S_{K}(\mathbf{L})$, then we apply again WSDM (as $b\in
S_{K}(\mathbf{L})$) and the assumption that $S_{K}(\mathbf{L})\cup T\left(
{\mathbf{L}}\right)  =L$, obtaining%
\[
0=\left(  b\wedge c\right)  {^{\sim}=b{^{\sim}\vee}c{^{\sim}=}b{^{\sim}\vee
}0=b{^{\sim}}}\text{,}%
\]
whereby $b=1$ since $b\in S_{K}(\mathbf{L})$.
\end{proof}

\begin{proposition}
\label{board1}Any directly indecomposable {$\mathbf{L}\in$}$\mathbb{V}_{1}$ is
either orthomodular or an antiortholattice.
\end{proposition}

\begin{proof}
Let {$\mathbf{L}$ be as in the statement of the proposition, and suppose that
$\mathbf{L}$ is directly indecomposable, but is neither orthomodular nor an
antiortholattice. }By Theorem \ref{centravanti}, the only PBZ*-central
elements of $\mathbf{L}$ are $0$ and $1$. By WSDM\ and Lemma \ref{goldrush}%
.(iii) we conclude that $0$ and $1$ are the only sharp elements $a$ such that
$b=\left(  b\wedge a{^{\sim}}\right)  \vee\left(  b\wedge\Diamond a\right)  $
for all $b\in L$.

Now, we want to show that $S_{K}(\mathbf{L})\cup T\left(  {\mathbf{L}}\right)
=L$. Let $x\in L$. The element $\Diamond\left(  x\wedge x^{\prime}\right)  $
is sharp and, by J2, we have that%
\[
b=\left(  b\wedge\left(  x\wedge x^{\prime}\right)  {^{\sim}}\right)
\vee\left(  b\wedge\Diamond\left(  x\wedge x^{\prime}\right)  \right)
=\left(  b\wedge\Diamond\left(  x\wedge x^{\prime}\right)  {^{\sim}}\right)
\vee\left(  b\wedge\Diamond\Diamond\left(  x\wedge x^{\prime}\right)  \right)
\]
for all $b\in L$. So, $\Diamond\left(  x\wedge x^{\prime}\right)  \in\left\{
0,1\right\}  $. If $\Diamond\left(  x\wedge x^{\prime}\right)  =0$, then
$x\wedge x^{\prime}\leq\Diamond\left(  x\wedge x^{\prime}\right)  =0$, whence
$x\in S_{K}(\mathbf{L})$. If $\Diamond\left(  x\wedge x^{\prime}\right)  =1$,
then
\[
{x{^{\sim}\leq}x{^{\sim}\vee\square x=}}\left(  x\wedge x^{\prime}\right)
{^{\sim}=0}\text{,}%
\]
and $x\in T\left(  {\mathbf{L}}\right)  $. Our claim is therefore settled.

Recall that $\mathbf{L}$ is directly indecomposable but fails to be an
antiortholattice --- whence by Corollary \ref{feccia}.(iii) there exist
$a,b\in L$ such that $a>\left(  a\wedge b{^{\sim}}\right)  \vee\left(
a\wedge\Diamond b\right)  $. Also, recall throughout the remainder of this
proof that $S_{K}(\mathbf{L})\cup T\left(  {\mathbf{L}}\right)  =L$. If
$b\notin S_{K}(\mathbf{L})$, then%
\[
a>\left(  a\wedge b{^{\sim}}\right)  \vee\left(  a\wedge\Diamond b\right)
=a\text{,}%
\]
a contradiction. Therefore $b\in S_{K}(\mathbf{L})$ and we can apply Lemma
\ref{totototo}: either $b=1$, or $b\leq x$ for every $x\notin S_{K}%
(\mathbf{L})$. If $b=1$, then $a>a$, a contradiction again. If there is some
$c\notin S_{K}(\mathbf{L})$, then $b\leq c$, whence $b={{\square b\leq\square
c=0}}$, which yields again the contradiction $a>a$. Therefore $L=S_{K}%
(\mathbf{L})$ and $\mathbf{L}$ is orthomodular, against our assumption.
\end{proof}

\begin{theorem}
\label{cacchiocacchio}

\begin{enumerate}
\item $\mathbb{V}_{1}\mathbb{=OML}{{\vee V}}\left(  \mathbb{AOL}\right)  $.

\item $\mathbb{V}_{2}\mathbb{=OML}\vee\mathbb{SAOL}$.

\item $\mathbb{V}_{3}\mathbb{=OML}\vee V\left(  \mathbf{D}_{3}\right)  $.
\end{enumerate}
\end{theorem}

\begin{proof}
(i) It will suffice to show that any subdirectly irreducible $\mathbf{L}%
\in\mathbb{V}_{1}$ is either an orthomodular lattice or an antiortholattice.
However, since $\mathbf{L}$ is directly indecomposable, Proposition
\ref{board1} applies and we obtain our conclusion.

(ii) Any subdirectly irreducible, and thus directly indecomposable, member of
$\mathbb{V}_{2}$ is either orthomodular, or an antiortholattice satisfying
SDM; since $\mathbb{SAOL}$ is generated by such antiortholattices, our claim follows.

(iii) This follows, as above, from the fact that $V\left(  \mathbf{D}%
_{3}\right)  $ is axiomatised by SK relative to ${{V}}\left(  \mathbb{AOL}%
\right)  $ \cite[Cor. 3.3]{PBZ2}.
\end{proof}

An upshot of this theorem is that $\mathbb{V}_{3}\subset\mathbb{V}_{2}%
\subset\mathbb{V}_{1}\subset\mathbb{SDM}\vee{{V}}\left(  \mathbb{AOL}\right)
$, where the last strict inclusion is witnessed by the PBZ$^{\ast}$ --lattice
$\mathbf{H}$ above which satisfies SDM, thus also WSDM,\ but fails J2, thus
showing in passing that $\mathbb{V}_{2}\subset\mathbb{SDM}$. Note, also, that
$\mathbb{V}_{3}$ is the unique cover of $\mathbb{OML}$ in the lattice of
subvarieties of $\mathbb{PBZL}^{\mathbb{\ast}}$, because any member of such
which is not included in $\mathbb{OML}$ contains $\mathbf{D}_{3}$ \cite[Thm.
5.5]{GLP1+}.

\begin{acknowledgement}
All authors gratefully acknowledge the following funding sources: the European
Union's Horizon 2020 research and innovation programme under the Marie
Sklodowska-Curie grant agreement No 689176 (project \textquotedblleft Syntax
Meets Semantics: Methods, Interactions, and Connections in Substructural
logics\textquotedblright); the project \textquotedblleft Propriet\`{a}
d`ordine nella semantica algebrica delle logiche non
classiche\textquotedblright, Regione Autonoma della Sardegna, L. R. 7/2007, n.
7, 2015, CUP: F72F16002920002. Discussions with Antonio Ledda and Matthew
Spinks were extremely helpful in shaping up this paper.
\end{acknowledgement}

\end{document}